\documentclass[11pt,reqno]{amsart}
\usepackage[utf8]{inputenc}
\usepackage{mathtools,graphicx,xcolor,float,minipage-marginpar,blindtext,amssymb,amsmath,amsthm,cleveref,enumitem,cjhebrew,titletoc}
\graphicspath{{image/}}
\usepackage{wrapfig}
\usepackage[mathscr]{euscript}
\usepackage[a4paper, hmargin=2cm, vmargin = 3cm]{geometry}

\usepackage{xpatch}
\makeatletter   
\xpatchcmd{\@tocline}
{\hfil\hbox to\@pnumwidth{\@tocpagenum{#7}}\par}
{\ifnum#1<0\hfill\else\dotfill\fi\hbox to\@pnumwidth{\@tocpagenum{#7}}\par}
{}{}
\makeatother    

\makeatletter
 \def\l@subsection{\@tocline{2}{0pt}{4pc}{6pc}{}}
\def\l@subsubsection{\@tocline{3}{0pt}{8pc}{8pc}{}}
 \makeatother

\allowdisplaybreaks[4]
\newcommand\norm[1]{\left\lVert#1\right\rVert}
\newcommand\ip[2]{\langle#1,#2\rangle}
\newcommand\ncm[1]{(\mathscr{A}_{#1}^{\dagger})_+}

\def\a{\alpha}

\def\p{\varphi}
\def\l{\lambda}

\def\g{\gamma}

\def\k{\kappa}

\def\s{\sigma}
\def\t{\tau}
\def\d{\delta}
\def\z{\zeta}
\def\th{\theta}

\def\D{\Delta}
\def\L{\Lambda}
\def\P{\Phi}

\def\ker{\mathrm{Ker}}

\def\rmd{\mathrm{d}}
\def\NN{\mathbb{N}}
\def\ZZ{\mathbb{Z}}

\def\RR{\mathbb{R}}
\def\CC{\mathbb{C}}
\def\DD{\mathbb{D}}
\def\BB{\mathbb{B}}
\def\EE{\mathbb{E}}
\def\TT{\mathbb{T}}
\def\FF{\mathbb{F}}
\def\HH{\mathbb{H}}
\def\nc{\mathrm{nc}}
\def\ac{\mathrm{ac}}
\def\rms{\mathrm{s}}
\def\cA{\mathcal{A}}
\def\cV{\mathcal{V}}
\def\cW{\mathcal{W}}

\def\cD{\mathcal{D}}

\def\cH{\mathcal{H}}
\def\cZ{\mathcal{Z}}
\def\scrA{\mathscr{A}}
\def\dju{\bigsqcup_{n=1}^{\infty}}

\newtheorem{theorem}{Theorem}[section]
\newtheorem{lemma}[theorem]{Lemma}
\newtheorem{corollary}[theorem]{Corollary}
\theoremstyle{definition}
\newtheorem{definition}[theorem]{Definition}
\newtheorem{remark}[theorem]{Remark}

\numberwithin{equation}{section}

\renewcommand{\tilde}{\widetilde}
\renewcommand{\hat}{\widehat}

\title{A Noncommutative Szeg\H{o}-Type Theorem on the Row-Ball}
\dedicatory{Dedicated to the memory of Professor Harry Dym \textcjheb{l}"\textcjheb{z}}
\date{\today}
\author[C. J. Gauntlett]{Connor J. Gauntlett}
\address{(CJG) School of Mathematics and Statistics\\
Newcastle University\\
Newcastle upon Tyne NE1 7RU UK}
\email{c.gauntlett@newcastle.ac.uk}

\author[D. P. Kimsey]{David P. Kimsey}
\address{(DPK) School of Mathematics and Statistics\\
Newcastle University\\
Newcastle upon Tyne NE1 7RU UK}
\email{david.kimsey@newcastle.ac.uk}
\subjclass[2020]{46L51, 46L52, 42C05}
\keywords{Noncommutative measures, Noncommutative function theory, Orthogonal polynomials, Verblunsky coefficients, Christoffel function, Szeg\H{o} limit theorem}

\begin{document}

\begin{abstract}
    In this paper we leverage the recently developed theory of noncommutative (nc) measures to prove a free noncommutative analogue of many known equalities extending the weak Szeg\H{o} limit theorem, by applying Constantinescu's theory of Schur parameters to an appropriate kernel on the free monoid on \(d\) generators, where \(d \geq 1\); in particular, we show that our nc Szeg\H{o} entropy depends only upon the absolutely continuous part of the associated nc measure. We obtain a correspondence between nc measures and multi-Toeplitz kernels arising from considering the moments of the nc measure, and apply this correspondence to study orthogonal polynomials associated to an nc measure. Finally, we study the determinantal zeros of those polynomials and obtain a noncommutative row-ball analogue of the so-called Zeros Theorem for orthogonal polynomials on the unit circle.
\end{abstract}

\maketitle

\tableofcontents

\section{Introduction}

The primary purpose of this paper is to prove a noncommutative analogue of Szeg\H{o}'s theorem, in different forms variously referred to as the Szeg\H{o}--Verblunsky theorem or the first, or weak, Szeg\H{o} limit theorem. In so doing, we introduce and study a number of tools reminiscent of the classical Szeg\H{o} theory including Verblunsky coefficients (or Schur parameters), orthogonal polynomials, and the Christoffel function associated to a measure.

The first form of Szeg\H{o}'s theorem we consider herein was first given by Verblunsky and hence is sometimes referred to as the Szeg\H{o}--Verblunsky theorem. This result links the absolutely continuous part of a (non-finitely atomic) measure \(\mu\) on the unit circle to the Verblunsky coefficients \((\g_n)_{n=0}^{\infty}\) of \(\mu\): let the Lebesgue decomposition of \(\mu\) with respect to the normalised Lebesgue measure \(\frac{\rmd\th}{2\pi}\) be \(\rmd \mu(\th) = w(\th) \frac{\rmd \th}{2\pi} + \rmd \mu_{\rms}(\th)\), and let \((\g_n)_{n=0}^{\infty} \subseteq \DD\) be the Verblunsky coefficients of \(\mu\). Then
\begin{equation}
    \label{eqn:SzegoVerblunsky}
    \prod_{j = 0}^{\infty} (1 - \lvert \g_j \rvert^2) = \exp \left( \int_{0}^{2\pi} \log\big(w(\th)\big) \; \frac{\rmd \th}{2\pi} \right),
\end{equation}
and hence in particular
\begin{equation}
    \label{eqn:SzegoCondition}
    \sum_{j=0}^{\infty} \lvert \g_j \rvert^2 < \infty \quad \text{ if and only if } \quad \int_{0}^{2\pi} \log\big(w(\th)\big) \; \frac{\rmd\th}{2\pi} > - \infty;
\end{equation}
we stress that while the Verblunsky coefficients of \(\mu\) depend on all of \(\mu\), the condition on the right-hand side of \eqref{eqn:SzegoCondition} --- the \emph{Szeg\H{o} condition} --- explicitly depends only upon the absolutely continuous part of \(\mu\).

In another form more often called the first, or weak, Szeg\H{o} limit theorem, one considers the Fourier coefficients \((c_n)_{n\in\ZZ}\) of \(\mu\), recalling that positivity of \(\mu\) implies that \(\overline{c_n} = c_{-n}\), and defines the Toeplitz determinants of \(\mu\) via
\begin{equation}
    \label{eqn:WeakSzegoLimit}
    D_n := \det \begin{bmatrix}c_0 & c_1 & \cdots & c_{n} \\ c_{-1} & c_0 & \cdots & c_{n-1} \\ \vdots & \vdots & \ddots & \vdots \\ c_{-n} & c_{-(n-1)} & \cdots & c_0\end{bmatrix}.
\end{equation}
One formulation of the weak Szeg\H{o} limit theorem then says that
\[
    \lim_{n\to\infty}\frac{D_n}{D_{n-1}} = \prod_{j = 0}^{\infty} (1 - \lvert \g_j \rvert^2) \left( = \exp \left( \int_{0}^{2\pi} \log\big(w(\th)\big) \; \frac{\rmd \th}{2\pi} \right) \right),
\]
providing another equivalent form for the quantity \(\prod_{j = 0}^{\infty} (1 - \lvert \g_j \rvert^2)\).

These results, their equivalent formulations and their consequences have far-reaching links across mathematical analysis; for a highly detailed discussion of much of this area of study and its history, see the book \cite{Sim05}. We highlight in particular the following summary theorem, \cite[Theorem 2.7.14]{Sim05}.

Let \(\mu\) be a non-finitely atomic measure on the circle \(\TT\) with Lebesgue decomposition with respect to the normalised Lebesgue measure \(\rmd\mu(\th) = w(\th)\frac{\rmd \th}{2\pi} + \rmd \mu_s \). Let \(\mu\) have associated Verblunsky coefficients \((\g_n)_{n=0}^{\infty}\), monic orthogonal polynomials \((\Phi_n)_{n=0}^{\infty}\), orthonormal polynomials \((\p_n)_{n=0}^{\infty}\), Toeplitz determinants \((D_n)_{n=0}^{\infty}\), Szeg\H{o} function \(D(z)\), Christoffel function \(\l_\infty(z; \rmd\mu)\), and Schur function \(f(z)\); let \(\k_n\) be the coefficient of \(z^n\) in \(\p_n\). For a polynomial \(p : \TT \to \CC\), denote its reverse polynomial by \(p^*(z) := z^{\rm{deg}(p)}\overline{p(1/\overline{z})}\). 

\begin{theorem}[Theorem 2.7.14 in \cite{Sim05}]
    \label{thm:CommSzego}
    The following quantities are all equal.
    \begin{enumerate}
        \item[\rm{(i)}] \(\exp(\int_{0}^{2\pi}\log(w(\th)) \; \frac{\rmd \th}{2\pi})\);
        \item[\rm{(ii)}] \(\lim_{n\to \infty} \norm{\P_n}^2 = \lim_{n\to\infty} \norm{\P_n^*}^2\);
        \item[\rm{(iii)}] \(\lim_{n\to\infty} \k_n^{-2}\);
        \item[\rm{(iv)}] \(\prod_{n=0}^\infty (1 - \lvert \g_n \rvert^2)\);
        \item[\rm{(v)}] \(\exp(\int_0^{2\pi} \log\left(1 - \lvert f(e^{i \th})\rvert^2\right) \frac{\rmd \th}{2\pi})\);
        \item[\rm{(vi)}] \(\lim_{n\to\infty} \frac{D_{n+1}}{D_n} = \lim_{n\to\infty} \sqrt[n]{D_n}\);
        \item[\rm{(vii)}] \(\l_\infty(0; \rmd\mu)\);
        \item[\rm{(viii)}] \(1 - \norm{Q_+ z^{-1}}^2, \text{ where \(Q_+\) is the projection onto \(\mathrm{span}\{z^j\}_{j=0}^\infty\)}\);
        \item[\rm{(ix)}] \(\lvert D(0) \rvert^2\);
        \item[\rm{(x)}] \(\lim_{n\to\infty} \lvert \p_n^*(0) \rvert^{-2}\);
        \item[\rm{(xi)}] \(\lim_{n\to\infty} \left(\sum_{j=0}^n \lvert \p_j(0) \rvert^2\right)^{-1}\);
        \item[\rm{(xii)}] The relative entropy \(S(\frac{\rmd \th}{2\pi} \; \vert \; \rmd \mu)\).
    \end{enumerate}
\end{theorem}

The Szeg\H{o}--Verblunsky theorem is the equality of (i) and (iv), while the weak Szeg\H{o} limit theorem equates (iv) and (vi), but this theorem highlights a plethora of the related quantities available as tools in the classical theory; in the sequel, we shall introduce noncommutative analogues of many of these objects and relate them in free noncommutative analogue of \Cref{thm:CommSzego}.

Free noncommutative (nc) functions are graded functions of (tuples of) square matrices of arbitrary size that respect intertwinings, originating in the work of J. L. Taylor \cite{Tay72}, \cite{Tay73} which was intended to develop a functional calculus for tuples of noncommuting operators. Building upon these works and the earlier vector-valued function theory of A. E. Taylor \cite{Tay37} and subsequently Hille and Phillips \cite[Chapter XXVI]{HP57}, nc function theory has seen a surge of activity in the past two decades, arising from several motivations: an incomplete list of examples includes holomorphic function theory \cite{Pop06b}, algebraic geometry \cite{Hel02}, control theory \cite{HMPV09}, free probability \cite{Voi10}, reproducing kernel Hilbert spaces \cite{BMV16} and functional calculus \cite{AMY20}. A review of the foundations of the field with a focus on generality and a noncommutative Taylor series can be found in \cite{KVV14}; likewise from a different perspective, the second part of \cite{AMY20} is dedicated to analogues of classical theorems of complex analysis in the free noncommutative setting. Moreover, \cite{KVV14} has an historical review between Chapter 9 and Appendix A, \cite[Chapter 11]{AMY20} provides several additional motivations for studying nc functions, and both of these discussions contain rich lists of references for further reading.

One recent development in the field is a theory of \emph{nc measures} which in many ways mimics the theory of measures on the unit circle in the complex plane, developed largely in a series of papers including \cite{JM17}, \cite{JM20}, \cite{JM22a}, \cite{JM22b}, and \cite{JM23}. Of particular note for us is the introduction of a Lebesgue decomposition for nc measures and the accompanying notions of absolutely continuous/singular nc measure, discussed in detail in \cite{JM22b}. As we work within this framework in large part, we shall begin the paper by recalling some of the relevant ideas.

Given the recent introduction of absolutely continuous nc measures, it becomes natural to ask whether an analogue of Szeg\H{o}'s theorem in some form holds with the noncommutative measures of Jury and Martin in place of a measure on the unit circle. Some work in this direction exists already: a phrasing of Szeg\H{o}'s theorem in terms of factorisations of symbols of Toeplitz operators admits a full generalisation to the setting of multi-Toeplitz operators on the full Fock space \cite[Theorem 1.3]{Pop06}; Jury and Martin's Lebesgue decomposition \cite{JM22b} shows that the nc measure-theoretic interpretation of such operators is as Radon--Nikodym derivatives, hence inducing absolutely continuous nc measures. A connection between the \emph{prediction entropy} of a multi-Toeplitz operator studied in \cite{Pop06} and our work is discussed in \Cref{rem:Popescu06}.

Meanwhile, in \cite{Pop01}, operator-valued multi-Toeplitz kernels are studied, and an accompanying notion of Verblunsky coefficients therein called generalised Schur sequences is introduced. In this setting, Popescu is able to obtain a weak Szeg\H{o} limit theorem relating the multi-Toeplitz determinants to these generalised Schur sequences; in \cite{Pop06} it is noted that the quantity studied in this weak Szeg\H{o} limit theorem is distinct from the prediction entropy of a multi-Toeplitz operator.

The Verblunsky coefficients (or sometimes \emph{Schur parameters}) \((\g_n)_{n=0}^{\infty} \subseteq \overline{\DD}\) associated to a measure \(\mu\) on \(\TT\) are perhaps the most ubiquitous tool in Szeg\H{o} theory. One way to generate this association is via the \emph{Schur algorithm} (or \emph{coefficient stripping}), a recursive algorithm utilising the power series representation of the Schur function associated to \(\mu\). Extensions of this algorithm to several variables have been studied: \cite{ABK02} develops a Schur algorithm for Schur functions on the unit ball in \(\CC^d\) where the Verblunsky coefficients are now elements of the unit ball, while in several noncommutative variables the aforementioned work \cite{Pop01} discusses generalised Schur sequences associated to a multi-Toeplitz kernel, realised as sequences of tuples of matrices; notably, in both of these cases, the multivariate analogue of the Verblunsky coefficients are vector- or operator-valued, rather than complex-valued and inside the unit disc.

In order to produce a complex-valued (in particular, \(\overline{\DD}\)-valued) analogue of Verblunsky coefficients for nc measures, then, a different method is required, and we discover our solution in \cite{CJ02b}, which in turn relies on the abstract kernel theory of \cite{Con96}. In \cite{CJ02b}, it is shown that kernels on \(d\) noncommuting variables (that is, maps \(\FF_d^+ \times \FF_d^+ \to \CC\)) satisfying certain technical conditions correspond to families of complex numbers in the unit disk, indexed now by the free monoid on \(d\) generators \(\FF_d^+\) rather than \(\NN\). 

Classically, of course, positive measures on the unit circle correspond bijectively with positive kernels on the natural numbers by defining \(K(n,m) := \int_0^{2\pi} e^{i(n-m)\th} \; \rmd \mu(\th)\): this is part of the Carath\'{e}odory--Toeplitz theorem (see, e.g., Theorem 11.3, \(\mathrm{(i)} \iff \mathrm{(v)}\) of \cite{Sch17}), so a noncommutative version of this correspondence would allow us to produce Verblunsky coefficients for an nc measure. We shall see that such a correspondence does hold, with the caveat that we need to slightly modify one of the axioms of Constantinescu and Johnson's definition --- however, this modification does not affect the theory thereafter, so we retain their results for the setting with which we proceed to work. 

We shall also require study of the \emph{Christoffel function} of a measure. If \(\mu\) is a non-finitely atomic measure on \(\TT\), \cite{Sim05} defines what we shall call the Christoffel approximates of \(\mu\) to be
\begin{equation}
    \label{eqn:CDApproximateDef}
    \l_{n}(\z, \rmd \mu) := \min\left\{\int_0^{2\pi} \lvert \pi(e^{i\th}) \rvert^2 \, \rmd \mu(\th) : \pi \in \CC[z], \deg(\pi) \leq n, \pi(\z) = 1 \right\}
\end{equation}
and the Christoffel function of \(\mu\) is then
\[
    \l_\infty (\z, \rmd\mu) := \lim_n \l_n(\z, \rmd\mu) = \inf\left\{\int_0^{2\pi} \lvert \pi(e^{i\th}) \rvert^2 \, \rmd \mu(\th) : \pi \in \CC[z], \pi(\z) = 1 \right\}.
\]
This function is linked to the Christoffel-Darboux kernel of the measure,
\[
    K_n(z,\z) = \sum_{j=0}^{n} \overline{\p_j(\z)}\p_j(z) \quad \text{ and } \quad K(z,\z) = \sum_{j=0}^{\infty} \overline{\p_j(\z)}\p_j(z)
\]
where \((\p_j)_{j=0}^{\infty}\) are the orthonormal polynomials associated to \(\mu\), by the formula
\begin{equation}
    \label{eqn:CommChristoffelFunction}
    \l_n(\z, \rmd\mu) = \frac{1}{K_n(\z,\z)}
\end{equation}
and it follows that
\[
    \l_\infty(\z, \rmd\mu) = \frac{1}{K(\z,\z)}.
\]
A noncommutative Christoffel function is discussed in \cite{BMV23} in the context of bounded tracial states on the algebra of noncommutative polynomials equipped with a particular involution with respect to which the monomials \(\{X_1, \ldots, X_d\}\) are self-adjoint. In \cite{BMV23}, the appropriate analogue of \eqref{eqn:CommChristoffelFunction} is taken as a definition and it is shown that a similar minimum form to \eqref{eqn:CDApproximateDef} holds, despite this quantity now being matrix-valued. This setting differs from ours, but motivated by this paper, we introduce a similar construction and obtain a similar result for what we name \emph{nc Christoffel approximates}. By taking a pointwise limit over the bounded polynomial degree in this minimum form, we may obtain an nc Christoffel function, and we show that this limit exists. Of course, orthogonal polynomials are themselves classically objects of much interest; as \cite{BMV23} remarks, noncommutative orthogonal polynomials have been a recent object of study, and we direct the interested reader there for further reading in this direction.

Once this is done, we are able to prove a noncommutative version of \Cref{thm:CommSzego} linking all of these ideas; a particularly notable consequence of this result is that a product analogous to the left-hand side of \eqref{eqn:SzegoVerblunsky} depends only on the absolutely continuous part of the nc measure. As in the classical theory, this statement can be rephrased in terms of the square-summability of the Verblunsky coefficients of an nc measure.

Finally, we pivot to a related problem, inspired by the \emph{Zeros Theorem} (see e.g. \cite[Theorem 1.7.1]{Sim05}) on the location of the zeros of the orthonormal polynomials \((\p_n)_{n=0}^{\infty}\) and their reverse polynomials \((\p_n^*)_{n=0}^{\infty}\). We obtain an analogous result for the noncommutative orthonormal polynomials associated to an nc measure, with the caveat that we restrict our discussion to more particular subsets of the nc space \(\CC^d_\nc\).

\subsection{Main Conclusions}

\begin{enumerate}
    \item[(C1)] We produce a correspondence between nc measures and kernels on \(\FF_d^+\) exhibiting multi-Toeplitz structure (see \Cref{thm:Correspondence}).
    \item[(C2)] We establish a minimisation formula for the nc Christoffel approximates of an nc measure (\Cref{lem:ChristoffelFunction}), and hence obtain an infimum formula for the nc Christoffel function (see \Cref{thm:CFExists}).
    \item[(C3)] Our main result is a noncommutative analogue of \Cref{thm:CommSzego} (see \Cref{thm:NCSzego}), though in our setting the list splits in two; this in particular shows that an analogue of the left hand side of \eqref{eqn:SzegoVerblunsky} depends only upon the absolutely continuous part of the nc measure.
    \item[(C4)] We establish an analogue of the Zeros Theorem on the row-ball for orthonormal polynomials associated to an nc measure (see \Cref{thm:NCZerosThm}).
\end{enumerate}

The paper is structured as follows. In Section 2, we briefly discuss the relevant parts of the nc measure theory of Jury and Martin. Section 3 establishes the correspondence (C1). In Section 4 we apply the ideas of \cite{CJ02b} to develop and study analogues for Verblunsky coefficients and quantities such as orthonormal/monic orthogonal polynomials or Toeplitz determinants in the setting of nc measures. In Section 5 we study the nc Christoffel approximates of an nc measure and establish (C2). Section 6 is devoted primarily to proving our main result (C3). Section 7 introduces determinantal zeros of a noncommutative polynomial to establish (C4).

\subsection{Indices and Notation}

Herein we list the main items of notation we refer to throughout the paper. We begin with a discussion of the free monoid on \(d\) generators, \(\FF_d^+\), which appears often in place of \(\NN\) in the free noncommutative setting. We use the convention of denoting the generators \(1, \ldots, d\) and the identity by \(\emptyset\); elements of \(\FF_d^+\) are called \emph{words}, and the monoid product is concatenation of words. The \emph{length} of a word \(\s \in \FF_d^+\), denoted \(\lvert \s \rvert\), is the number of generators used to form \(\s\): for example, \(\lvert 112\rvert = 3\), and by convention we use \(\lvert \emptyset \rvert = 0\). Given a tuple \(Z = (Z_1, \ldots, Z_d)\), of operators or of formal variables, we shall use the notation
\[
    Z^\s := Z_{\s_1} \cdot \cdots \cdot Z_{\s_n}
\]
whenever we write the word \(\s\) as \(\s = \s_1 \cdots \s_n \in \FF_d^+\) with \(\s_1, \ldots, \s_n \in \{1, \ldots, d\}\). In other words, we write out the word \(\s\) with \(Z_j\) in place of each generator \(j\) appearing in \(\s\) --- for example, \((Z_1, Z_2)^{112} = Z_1 Z_1 Z_2 = Z_1^2 Z_2\). We take the convention that \(Z^{\emptyset} = 1\).

We shall also have cause to refer to the \emph{short lexicographical} (shortlex) ordering on \(\FF_d^+\); when \(\s\) precedes \(\t\) in this ordering, we write \(\s \prec \t\). Explicitly, this is the ordering obtained by declaring for generators that \(1 \prec \cdots \prec d\), that \(\s \prec \t\) if \(\lvert \s \rvert < \lvert \t \rvert\), and that if \(\lvert \s \rvert = \lvert \t \rvert = n\), we say that \(\s = \s_1\cdots\s_n \prec \t_1\cdots\t_n = \t\) if there exists \(j = 1, \ldots, n\) such that
\[
    \s_1 = \t_1, \ldots, \s_{j-1} = \t_{j-1} \text{ and } \s_j \prec \t_j.
\]
For example, when \(d = 2\), the shortlex ordering on the first few elements of \(\FF_2^+\) is
\[
    \emptyset \prec 1 \prec 2 \prec 11 \prec 12 \prec 21 \prec 22 \prec 111 \prec 112 \prec \ldots.
\]
We denote by \(\s(n)\) the final word under the shortlex ordering of length \(n\), which is the generator \(d\) repeated \(n\) times; the immediate predecessor of \(\s \in \FF_d^+\) with respect to the shortlex ordering is denoted \(\s - 1\). For example, when \(d=2\), \(\s(1) = 2\), \(\s(3) = 222\), \(112 - 1 = 111\) and \(11 - 1 = 2\).

Of course, when \(d = 1\), we have \(\FF_1^+ = \{\emptyset, 1, 11, \ldots \} = \{1^{j} : j = 0, 1, 2, \ldots\} \cong \NN\), the shortlex ordering is the usual ordering on \(\NN\), and \(\s(n) = n\) is the only word of length \(n\).

\bigskip

The remainder of this section contains a list of commonly-used notation for the reader's convenience.

\bigskip

\begin{itemize}
    \item \(\DD, \overline{\DD}, \TT\) - the open and closed unit discs and the unit circle in the complex plane, respectively.
    \item \(\s, \t\, \a\) - words in \(\FF_d^+\).
    \item \(H^2(\DD)\) - the Hardy space on the unit disc.
    \item \(A(\DD)\) - the disc algebra.
    \item \(C(\TT)\) - the Banach space of continuous functions \(\TT \to \CC\) with supremum norm \(\norm{\cdot}_\infty\).
    \item \(\CC^d_{\nc}\) - the nc space in \(d\) variables; see Section 2.1.
    \item \(\BB_{\nc}^d\) - the row-ball in \(d\) variables; see Section 2.1.
    \item \(\BB^d_n\) - the unit ball in \((\CC^{n\times n})^d\), i.e. the \(n^{\text{th}}\) level of the nc set \(\BB^d_\nc\).
    \item \(\mathrm{Ext}(\overline{\BB}^d_{\nc})\) - the set taking the place of \(\CC\setminus\overline{\DD}\) in our study of determinantal zeros of orthogonal polynomials; see \Cref{lem:ZerosInside}.
    \item \(\partial \BB_{\nc}^d\) - the distinguished boundary of the row-ball; see Section 7.
    \item \(\HH^2_d\) - the nc Hardy space in \(d\) variables, i.e. the full Fock space on \(\CC^d\); see Section 2.1.
    \item \(L = (L_1, \ldots, L_d)\) - the left free shift on \(\HH^2_d\); see \Cref{def:LeftShift}.
    \item \(\ncm{d}\) - the space of (positive) nc measures in \(d\) variables; see \Cref{def:NCMeasure}.
    \item \(\mu, \mu_{\ac}, \mu_\rms\) - either a measure on the unit circle or an nc measure, and respectively its absolutely continuous and singular parts; see Section 2.2.
    \item \(T\) - a closed, densely-defined, positive \(L\)-Toeplitz operator serving as the nc Radon--Nikodym derivative of an nc measure; see \Cref{def:leftToeplitzOperator}.
    \item \(K\) - a multi-Toeplitz kernel on \(\FF_d^+\) with entries \(K(\s,\t)\) for \(\s,\t \in \FF_d^+\); see \Cref{def:multiToeplitzKernel}.
    \item \(K_\mu, \mu_K\) - given a multi-Toeplitz kernel \(K\), \(\mu_K\) is the nc measure corresponding to \(K\); given an nc measure \(\mu\), \(K_\mu\) is the multi-Toeplitz kernel corresponding to \(\mu\); see \Cref{thm:Correspondence}.
    \item \((\g_{\s,\t})_{\s,\t \in \FF_d^+}\) - a sequence associated to a multi-Toeplitz kernel by \cite{CJ02b} there called the Schur parameters of the kernel.
    \item \((\g_{\s})_{\s \in \FF_d^+}\) - the Verblunsky coefficients of an nc measure (or multi-Toeplitz kernel) given by \(\g_{\s} := \g_{\emptyset,\s}\); see \Cref{def:VerblunskyCoefficients}.
    \item \(d_\s\) - the defect of the Verblunsky coefficient \(\g_\s\), \(d_\s := \sqrt{1 - \lvert \g_\s \rvert^2}\).
    \item \((D_\s)_{\s\in\FF_d^+}\) - the multi-Toeplitz determinants of the multi-Toeplitz kernel \(K\), for \(\s_0 \in \FF_d^+\) given by \(D_{\s_0} := \det \begin{bmatrix} K(\s, \t) \end{bmatrix}_{\emptyset \preceq \s,\t \preceq \s_0}\).
    \item \(\CC[z]\) - the space of polynomials in one variable \(z\) with complex coefficients.
    \item \(\CC[z]_n\) - the space of polynomials in one variable \(z\) with complex coefficients of degree at most \(n\).
    \item \(\CC\langle Z \rangle = \CC \langle Z_1, \ldots, Z_d \rangle\) - the space of polynomials in noncommuting indeterminates \(Z = (Z_1, \ldots, Z_d)\), i.e. finite \(\CC\)-linear combinations of elements of the set \(\{Z^{\s} : \s \in \FF_d^+\}\); elements of this space are called noncommutative (nc) polynomials.
    \item \(\CC\langle Z \rangle_\s\) - the space of nc polynomials \(\sum_{\t \in \FF_d^+} c_\t Z^{\t}\) such that \(c_\t = 0\) whenever \(\t \succ \s\) in the shortlex ordering.
    \item \(\langle \cdot, \cdot \rangle_K\) and \(\langle \cdot , \cdot \rangle_\mu\) - the inner product induced by a multi-Toeplitz kernel \(K\) and an nc measure \(\mu\), respectively.
    \item \(\norm{\cdot}_K, \norm{\cdot}_\mu\) - the norms induced by the inner products above.
    \item \(L^2(K), L^2(\mu)\) - the closure of \(\CC\langle Z_1, \ldots, Z_d\rangle\) with respect to \(\norm{\cdot}_K, \norm{\cdot}_\mu\) respectively.
    \item \((\p_\s)_{\s \in \FF_d^+}, (\P_\s)_{\s\in\FF_d^+}\) - respectively the orthonormal and monic orthogonal polynomials associated to an nc measure.
    \item \((\p_\s^{\#})_{\s \in \FF_d^+}, (\P_\s^{\#})_{\s\in\FF_d^+}\) - respectively the reverse polynomials of the orthonormal and monic orthogonal polynomials associated to an nc measure; see \Cref{def:reversePolynomials}.
    \item \(\a_{\s,\t}\) - the coefficient of \(Z^\t\) in \(\p_\s\).
    \item \(\k_{\mu,n}\) - the \(n^{\text{th}}\) nc Christoffel-Darboux kernel associated to an nc measure \(\mu\); see \Cref{def:ncCDK}.
    \item \(\l_n(\cdot; \rmd \mu)\) - the \(n^{\text{th}}\) classical Christoffel approximate associated to a measure \(\mu\) on the unit circle.
    \item \(\l_{\infty}(\cdot; \rmd \mu)\) - the classical Christoffel function associated to a measure \(\mu\) on the unit circle.
    \item \(\L_n(\cdot; \mu)\) - the \(n^{\text{th}}\) nc Christoffel approximate associated to an nc measure \(\mu\); see \Cref{def:ncCDK}. 
    \item \(\L_\infty(\cdot; \mu)\) - the nc Christoffel function associated to an nc measure \(\mu\); see \Cref{def:ncChristoffelFunction}.
    \item \(Q_n\) - the minimising polynomial for the minimum form of \(\L_n(\cdot; \mu)\); see \Cref{rem:Minimiser}.
    \item \(\cZ(P)\) - the set of determinantal zeros of an nc polynomial \(P\); see \Cref{def:DeterminantalZero}.
    \item \(0^{(d)}_{1\times1}\) - the \(d\)-tuple of scalar zeros \((0, \ldots, 0) \in \CC^d\).
\end{itemize}

\section{Background: NC Measures}

\subsection{Motivation and Definitions}

Herein we briefly recall some relevant notions and definitions; much more detailed exposition on this topic can be found in, for example, Sections 2 and 3 of \cite{JM22a}, which we follow for much of this section.

We begin with a particular construction of some classical objects. Let \(H^2(\DD)\) be the Hardy space on the disc, that is, the holomorphic function space
\[
    H^2(\DD) = \{f(z) = \sum_{n = 0}^{\infty} c_n z^n : \sum_{n=0}^{\infty} \lvert c_n \rvert^2 < \infty \},
\]
and let \(S : H^2(\DD) \to H^2(\DD)\) be the forward shift operator given by \((Sf)(z) = zf(z)\). Furthermore, let \(A(\DD)\), the disc algebra, be the Banach algebra
\[
    A(\DD) = \left\{f: \overline{\DD} \to \CC : f \text{ is coontinuous}  \; {\rm and} \; f\vert_{\DD} \text{ holomorphic}\right\}
\]
equipped with the supremum norm.

Firstly, note that \(S\) does not satisfy any polynomial identities, and has operator norm \(\norm{S}_{\mathrm{op}} = 1\), so that the map
\[
    \CC[z] \to \{p(S) : p \in \CC[z] \}, \quad z \mapsto S
\]
is a bijective isometry. It is well-known that the disc algebra is the supremum-norm closure of \(\CC[z]\), and therefore can be realised as the operator-norm closure of the set of polynomials in \(S\):
\[
    A(\DD) \cong \overline{\left\{ p(S) : p \in \CC[z] \right\}}^{\norm{\cdot}_{\mathrm{op}}},
\]
that is, \(A(\DD)\) is the operator-norm-closure of the algebra generated by the identity and \(S\).

By the Stone--Weierstrass theorem, any continuous, complex-valued function on the unit circle \(\TT\) can be uniformly approximated by polynomials in \(z\) and \(\overline{z}\) and therefore by elements of \(A(\DD) + A(\DD)^*\) (where \(A(\DD)^* = \{\overline{f} : f \in A(\DD)\}\)). From this, it follows that
\[
    C(\TT) = \overline{A(\DD) + A(\DD)^*}^{\norm{\cdot}_{\mathrm{op}}}.
\]
By the Riesz--Markov--Kakutani representation theorem, any (positive) finite, regular Borel measure on \(\TT\) induces a (positive) bounded linear functional on \(C(\TT)\) and vice versa. Hence elements of the dual of \(\overline{A(\DD) + A(\DD)^*}^{\norm{\cdot}_{\mathrm{op}}}\) and finite, regular Borel measures on \(\TT\) are in bijection.

The noncommutative setup generalises this construction. We begin with the \(d\)-variate noncommutative stand-in for the complex plane, the \emph{nc space} 
\[
    \CC^d_{\nc} := \dju (\CC^{n\times n})^d
\]
consisting of tuples of square matrices of arbitrary size, and for the unit disc, the \emph{row-ball}
\[
    \BB^d_{\nc} := \dju \left\{Z = (Z_1, \ldots, Z_d) \in (\CC^{n\times n})^d : Z_1 Z_1^* + \cdots + Z_d Z_d^* < I_n \right\},
\]
consisting of tuples of strict row contractions. The row-ball is a noncommutative set in the sense of \cite{KVV14} with matrix ``levels" or ``windows"
\[
    \BB^d_n = \left\{Z = (Z_1, \ldots, Z_d) \in (\CC^{n\times n})^d : Z_1 Z_1^* + \cdots + Z_d Z_d^* < I_n \right\}.
\]
A foundational result of Popescu \cite{Pop06b} says that any formal power series in \(d\) noncommuting variables with square-summable complex coefficients, that is any
\[
    \sum_{\s \in \FF_d^+} c_\s Z^\s \quad \text{ such that } \quad \sum_{\s \in \FF_d^+} \lvert c_\s \rvert^2 < \infty,
\]
converges absolutely (in operator norm) on the row-ball, and uniformly on the subset \(r\BB^d_\nc\) for any \(0 < r < 1\). Such a power series can therefore be viewed as a locally bounded (hence holomorphic) nc function on the row-ball \cite[Chapter 8]{KVV14}. On the other hand, any locally bounded nc function on the row-ball has such a power series representation (the \emph{Taylor-Taylor series}) \cite[Chapter 7]{KVV14}, so that the Hilbert space
\[
    \HH^2_d := \{ f(Z) = \sum_{\s \in \FF_d^+} c_\s Z^\s : \sum_{\s \in \FF_d^+} \lvert c_\s \rvert^2 < \infty \}
\]
of all such power series is a noncommutative version of \(H^2(\DD)\). Indeed, \(\HH^2_1 \cong H^2(\DD)\) so that in one variable the noncommutative function theory coincides with the classical function theory.

\begin{remark}
    The Hardy space \(H^2(\DD)\) arises variably as the space of holomorphic functions \(\DD \to \DD\) with square-summable Taylor series coefficients, the reproducing kernel Hilbert space corresponding to the Szeg\H{o} kernel
    \[
        k_S(z,w) := \frac{1}{1-z\overline{w}} = \sum_{n=0}^{\infty} z^n \overline{w^n}, \quad z,w \in \DD,
    \]
    and as the holomorphic function space isomorphic as a Hilbert space to \(\ell^2(\NN)\). The nc Hardy space \(\HH^2_d\) similarly arises as the space of holomorphic nc functions described above, as the noncommutative reproducing kernel Hilbert space (in the sense of \cite{BMV16}) corresponding to the nc Szeg\H{o} kernel
    \[
        K_S(Z,W)[\cdot] := \sum_{\s \in \FF_d^+} Z^\s [\cdot] (W^\s)^*, \quad Z, W \in \BB^d_{\nc},
    \]
    and as the full Fock space on \(\CC^d\),
    \[
        F^2(\CC^d) := \bigoplus_{k=0}^{\infty} \left(\CC^d\right)^{\otimes k},
    \]
    which is naturally isomorphic to \(\ell^2(\FF_d^+)\) by sending the basis element \(e_{\s_1} \otimes \cdots \otimes e_{\s_n} \in F^2(\CC^d)\) to \(e_{\s_1\cdots\s_n} \in \ell^2(\FF_d^+)\) for \(\s_1 \cdots \s_n \in \FF_d^+\). A detailed discussion of the different origins of \(\HH^2_d\) with further references can be found in, for example, \cite{JM21}.
\end{remark}

\begin{definition}
    \label{def:LeftShift}
    Let \(\HH^2_d\) be the nc Hardy space. The \emph{left free shift} is the \(d\)-tuple
    \[
        L = (L_1, \ldots, L_d),
    \]
    where for \(j = 1, \ldots, d\), \(L_j\) is the \(j^{\text{th}}\) (left) shift operator on \(\HH^2_d\) given by
    \[
        (L_j f)(Z) = Z_j f(Z).
    \]
\end{definition}

\begin{remark}
    The left free shift's components \(L_j, j = 1, \ldots, d,\) are each isometries on \(\HH_d^2\), with pairwise orthogonal ranges:
    \[
        L_j^* L_k = \d_{jk} I_{\HH^2_d}, \quad 1 \leq j,k \leq d;
    \]
    these conditions together make \(L\) a \emph{row-isometry}, an isometry \((\HH^2_d)^d \to \HH^2_d\). This property will be central to our study in particular in Section 3.
\end{remark}

Emulating the construction outlined above, the \emph{nc disc algebra} is the operator norm-closure of the space of noncommutative polynomials evaluated on the left free shift:
\[
    \cA_d := \overline{\{p(L) : p \in \CC\langle Z_1, \ldots, Z_d \rangle\}}^{\norm{\cdot}_{\mathrm{op}}};
\]
similarly, this is the norm-closure of the algebra generated by the identity operator and the components of \(L\).

We next construct the \emph{nc disc system} from \(\cA_d\) and the space of adjoints of elements of \(\cA_d\): defining \(\cA_d^* := \{f^* : f \in \cA_d\}\) where \(f^* : Z \mapsto f(Z)^*\), the nc disc system is the operator system
\[
    \scrA_d := \overline{\cA_d + \cA_d^*}^{\norm{\cdot}_{\mathrm{op}}}.
\]
Notice that \(\scrA_1 \cong C(\TT)\) and recall that positive finite regular Borel measures on \(\TT\) are in bijection with positive elements of \(C(\TT)^{\dagger}\), the dual of \(C(\TT)\). Motivated by this, we define an nc measure as follows.

\begin{definition}
    \label{def:NCMeasure}
    A \emph{noncommutative (nc) measure} \(\mu\) is a positive functional on the nc disc system, i.e. a linear map
    \[
        \mu : \scrA_d \to \CC
    \] 
    such that \(\mu(f) \geq 0\) for all \(f \in \scrA_d\) such that \(f \geq 0\). We write \(\mu \in \ncm{d}\).
\end{definition}

\subsection{The Lebesgue Decomposition and the Radon--Nikodym Derivative}

Identify \(H^2(\DD)\) with the subspace \(H^2(\TT)\) of \(L^2(\TT)\) by identifying \(f \in H^2(\DD)\) with the function \(\tilde{f} \in L^2(\TT)\) given by
\[
    \tilde{f}(\z) = \lim_{r \to 1} f(r\z),
\]
which is defined almost everywhere \cite[Theorem 3.3.8]{Kat04}.

An operator \(T : H^2(\DD) \to H^2(\DD)\) is called \emph{Toeplitz} with symbol \(g \in L^{\infty}(\TT)\) if
\[
    T = PM_g \vert_{H^2(\TT)},
\]
where \(P\) is the projection onto \(H^2(\TT)\) and \(M_g\) is the multiplication operator induced by \(g\). A famous theorem of Brown and Halmos \cite[Theorem 6]{BH64} says that \(T\) is Toeplitz if and only if the \emph{Brown--Halmos identity},
\[
    S^* T S = T
\]
holds, where \(S\) is the forward shift operator on \(H^2(\DD)\). This has a natural extrapolation to \(d\) noncommuting variables: the following definition was given by Popescu and has subsequently been studied extensively, e.g. in \cite{Pop06}.

\begin{definition}
    \label{def:multiToeplitzOperator}
    Let \(T : \HH^2_d \to \HH^2_d\) be a bounded positive operator. We say \(T\) is \emph{multi-Toeplitz} if
    \[
        L_j^* T L_k = \d_{jk} I_{\HH^2_d}, \quad 1 \leq j,k \leq d,
    \]
    i.e. if \(T\) satisfies \(d\) separate Brown--Halmos identities.
\end{definition}

In \cite{JM22a}, \cite{JM22b}, Jury and Martin study a related class of operators.

\begin{definition}[\cite{JM22b}, Definition 4.4]
    \label{def:leftToeplitzOperator}
    A closed, positive, densely-defined operator \(T: \cD(T) \subseteq \HH^2_d \to \HH^2_d\) is called \emph{left-} or \emph{\(L\)-Toeplitz} if
    \begin{enumerate}
        \item \(\cD(\sqrt{T})\) is \(L\)-invariant, i.e. \(L_j \cD(\sqrt{T}) \subseteq \cD(\sqrt{T})\) for \(j = 1, \ldots, d\), and
        \item the associated closed quadratic form
        \[
            q_T(g,h) := \langle \sqrt{T} g, \sqrt{T} h \rangle_{\HH^2_d}, \quad g,h\in \cD(\sqrt{T})
        \]
        satisfies
        \[
            q_T(L_j g, L_k h) = \d_{jk} q_T(g,h), \quad g,h \in \cD(\sqrt{T}).
        \]
    \end{enumerate}
\end{definition}

\begin{remark}
    The connection between these notions can be found in \cite[Remark 6]{JM22a}: if \(T : \HH^2_d \to \HH^2_d\) is bounded, then \(T\) is \(L\)-Toeplitz if and only if \(T\) is multi-Toeplitz. In this sense, one may consider \(L\)-Toeplitz operators an unbounded generalisation of multi-Toeplitz operators.
\end{remark}

Building on the Lebesgue decomposition for (positive semi-definite) quadratic forms of Simon \cite{Sim78}, Jury and Martin prove a decomposition (with respect to the nc Lebesgue measure) for nc measures \cite[Theorem 4]{JM22a}: if \(\mu \in \ncm{d}\), there exists a unique decomposition
\[
    \mu = \mu_{\ac} + \mu_{\rms}
\]
where \(\mu_{\ac}, \mu_{\rms} \leq \mu\) are nc measures, \(\mu_{\ac}\) is the maximal absolutely continuous (with respect to the nc Lebesgue measure) nc measure bounded above by \(\mu\), and \(\mu_\rms\) is a singular (again with respect to the nc Lebesgue measure) nc measure; here a functional \(\nu\) is absolutely continuous (resp. singular) if the quadratic form \(q_\nu(a_1, a_2) := \nu(a_1^* a_2)\) \cite[Equation (4.3)]{JM22a} on \(\cA_d \times \cA_d\) is absolutely continuous (resp. singular) in the sense of \cite{Sim78}.

Moreover, \cite[Theorem 5, Lemma 2]{JM22a} if \(\mu\) is an absolutely continuous nc measure, then there exists a unique closed, densely-defined positive \(L\)-Toeplitz operator \(T : \cD(T) \subseteq \HH^2_d \to \HH^2_d\) such that \(\cA_d \subseteq \cD(\sqrt{T})\) and \cite[Equation (6.1)]{JM22a}
\begin{equation}
    \label{eqn:NCRadonNikodym}
    \mu_{\ac}(a_2^*a_1) = \langle \sqrt{T}a_1, \sqrt{T}a_2 \rangle_{\HH^2_d}, \quad a_1, a_2 \in \cA_d.
\end{equation}
\begin{definition}
    \label{def:NCRadonNikodym}
    Let \(\mu \in \ncm{d}\) be an nc measure. The \emph{nc Radon--Nikodym derivative} of \(\mu\) is the closed, densely-defined positive \(L\)-Toeplitz operator \(T\) given by \eqref{eqn:NCRadonNikodym}, which will in general be unbounded.
\end{definition}

As Jury and Martin elaborate in \cite[Section 6]{JM22a}, this operator \(T\) can be thought of as a noncommutative Radon--Nikodym derivative for the nc measure \(\mu\), see the discussion preceding Theorem 6.2. In this way, we may think of closed, densely-defined \(L\)-Toeplitz operators on \(\HH^2_d\) as corresponding to absolutely continuous nc measures; hence the theory e.g. surrounding the factorisation results relevant to Szeg\H{o}'s theorem in \cite{Pop06} can be re-cast in terms of absolutely continuous nc measures.

In the next section we offer another relationship along these lines, associating to any nc measure a multi-Toeplitz \emph{kernel} on \(\FF_d^+\), though this kernel may no longer define an unbounded operator on \(\HH^2_d\).

\subsection{The Univariate Setting}

We stress that this function theory is a genuinely multivariate phenomenon.

When \(d = 1\), the nc Hardy space \(\HH^2_1\) coincides with the usual Hardy space on the disc:
\[
    \HH^2_1 = \{ f(Z) = \sum_{\s \in \FF_1^+} c_\s Z^\s : \sum_{\s \in \FF_1^+} \lvert c_\s \rvert^2 < \infty \} = \{ f(Z) = \sum_{n=0}^{\infty} c_n Z^n : \sum_{n=0}^{\infty} \lvert c_n \rvert^2 < \infty \} = H^2(\DD).
\]
It follows that any nc function on the row-ball in one variable, \(\BB^1_\nc\), is completely determined by its behaviour on the scalar level \(\BB^1_1 = \DD\), and the holomorphic function theory in this case reduces to the classical theory.

Moreover, in a similar fashion, recall that by the Stone--Weierstrass theorem \(\scrA_1 \cong C(\TT)\). Hence nc measures in one variable are simply positive functionals on \(C(\TT)\), and by the Riesz--Markov--Kakutani representation theorem, such functionals correspond precisely to finite, positive, regular Borel measures on \(\TT\) so that again we recover the classical theory.

We remark that Verblunsky coefficients, orthogonal polynomials, and Szeg\H{o} limit theorems have all been studied in the univariate, matrix-valued setting, see e.g. \cite{DPS08}, \cite{DK16}. However, in this setting, the candidate measures are no longer scalar-valued but rather matrix-valued, and so this theory generalises the classical Szeg\H{o} theory in a different direction to that of nc measures.

\section{The Correspondence Between NC Measures and Multi-Toeplitz Kernels}

\begin{definition}
    Generalising the case where \(d = 1\), we say that a kernel \(K : \FF_d^+ \times \FF_d^+ \to \CC\) is \emph{positive} if the matrix \(K_{\s_0} := [K(\s,\t)]_{\s,\t \preceq \s_0}\) is positive semi-definite for all \(\s_0 \in \FF_d^+\). For example, if \(d = 2\), then
    \[
        \FF_2^+ = \{\emptyset, 1, 2, 11, 12, 21, 22, 111, 112, 121, \ldots\}
    \]
    (here we have listed the elements of \(\FF_2^+\) according to the shortlex ordering) and so e.g.:
    \begin{gather*}
        K_{\emptyset} = \begin{bmatrix}K(\emptyset, \emptyset)\end{bmatrix}, \\
        K_{1} = \begin{bmatrix}K(\emptyset, \emptyset) & K(\emptyset, 1) \\ K(1, \emptyset) & K(1,1)\end{bmatrix},  \\
        K_{2} = \begin{bmatrix}K(\emptyset, \emptyset) & K(\emptyset, 1) & K(\emptyset, 2)\\ K(1, \emptyset) & K(1,1) & K(1,2) \\ K(2, \emptyset) & K(2,1) & K(2,2)\end{bmatrix}, \\
        \text{and } K_{11} = \begin{bmatrix}K(\emptyset, \emptyset) & K(\emptyset, 1) & K(\emptyset, 2) & K(\emptyset, 11)\\ K(1, \emptyset) & K(1,1) & K(1,2) & K(1,11) \\ K(2, \emptyset) & K(2,1) & K(2,2) & K(2,11) \\ K(11, \emptyset) & K(11, 1) & K(11,2) & K(11,11)\end{bmatrix}.
    \end{gather*}
\end{definition}

In \cite{CJ02b}, the authors study a class of positive kernels \(K\) indexed by elements of \(\FF_d^+\), satisfying two additional axioms:
\begin{equation}
    \label{CJAx1}
    K(\t \s, \t\s') = K(\s,\s'), \quad \t, \s, \s' \in \FF^+_d,
\end{equation}
\begin{equation}
    \label{CJAx2}
    K(\sigma, \tau) = 0 \quad \text{ if there is no \(\a \in \FF_d^+\) such that \(\s = \a\t\) or \(\t = \a\s\).}
\end{equation}
The former, \eqref{CJAx1}, is a generalisation of the notion of Toeplitzness, or of shift-invariance: in one variable, this becomes
\[
    K(n+k, m+k) = K(n,m)
\]
so that our kernel is constant along each superdiagonal. On the other hand, the condition for \eqref{CJAx2} (and indeed for \eqref{CJAx2'} below) never occurs in one variable, as there is always some \(k \in \NN\) such that \(n = m + k\) or \(m = n + k\), so this phenomenon is genuinely multivariate.

Given a positive kernel satisfying \eqref{CJAx1} and \eqref{CJAx2}, \cite{CJ02b} goes on to derive sets of Verblunsky coefficients and polynomials which are orthonormal with respect to the kernel, both of which are also indexed by noncommutative words. This approach, based in the matrix theory of \cite{Con96}, reproduces the classical theory of Verblunsky coefficients and orthogonal polynomials when \(d = 1\).

Meanwhile, Jury and Martin (in, e.g., \cite{JM22a}, \cite{JM22b}) recently established a theory of noncommutative measures which also reduces to the classical theory in one variable. When \(d = 1\), kernels and measures enjoy a natural one-to-one correspondence, which leads one naturally to question whether this still holds in several noncommuting variables. As we shall see, nc measures and the kernels studied by \cite{CJ02b} do not correspond directly, but nc measures \emph{do} correspond to a very similar class of kernels. Motivated by \Cref{thm:Correspondence}, we shall be interested instead in so-called \emph{multi-Toeplitz} kernels.

\begin{definition}
    \label{def:multiToeplitzKernel}
    We say that a positive kernel \(K : \FF_d^+ \times \FF_d^+ \to \CC\) is \emph{multi-Toeplitz} if \(K\) satisfies the axioms
    \begin{equation}
    \label{CJAx1'}
        K(\t \s, \t\s') = K(\s,\s'), \quad \t, \s, \s' \in \FF^+_d,
    \end{equation}
    \begin{equation}
    \label{CJAx2'}
        K(\sigma, \tau) = 0 \quad \text{ if there is no \(\a \in \FF_d^+\) such that \(\s = \t\a\) or \(\t = \s\a\).}
    \end{equation}
\end{definition}

\begin{definition}
    \label{def:multiToeplitzDeterminants}
    Let \(K\) be a multi-Toeplitz kernel. The \emph{multi-Toeplitz determinants} of \(K\) are the determinants
    \[
        D_{\s_0} := \det \begin{bmatrix}
            K(\s, \t)
        \end{bmatrix}_{\s,\t \preceq \s_0} = \det K_{\s_0}
    \]
    for \(\s_0 \in \FF_d^+\).
\end{definition}

Multi-Toeplitz kernels are, of course, very similar to those of \cite{CJ02b}, with only a ``twisted" version of \eqref{CJAx2}, and so almost all of the theory established therein will apply also to these modified kernels. Critically, though, we shall see that this class corresponds precisely to that of nc measures in exactly the way measures and kernels correspond in one variable (indeed, in one variable, our correspondence reduces precisely to the classical one). 

\begin{remark}
   As discussed across the previous two sections, multi-Toeplitz kernels have been studied by a number of authors: they are the focus of numerous works of Popescu such as \cite{Pop01} and \cite{Pop06}, and (when they define bounded operators on \(\HH^2_d\)) they appear in Jury and Martin's works \cite{JM22a}, \cite{JM22b}, where they are called \emph{L-Toeplitz}. Constantinescu and Johnson have also studied multi-Toeplitz kernels, mentioning them explicitly in \cite{CJ02a} (though defined via \eqref{CJAx2} rather than \eqref{CJAx2'}). 
\end{remark}

\begin{remark}
    \label{rem:PopescuWeakSzego}
   We recall that \cite{Pop01} also considers objects associated to a multi-Toeplitz kernel analogous to the classical Verblunsky coefficients --- there dubbed \emph{generalised Schur sequences} --- but that this construction differs from the perspective presented herein in several key ways. In particular, the generalised Schur sequences of \cite{Pop01} are genuine sequences (that is, indexed by \(\NN\)) of tuples of matrices, whereas the Verblunsky coefficients we consider here are more generally a family (indexed by \(\FF_d^+\)) of single complex numbers inside the unit disc.
\end{remark}

We formulate the first direction of our correspondence as follows.

\begin{lemma}
    \label{lem:Correspondence1}
    Let \(\mu \in \ncm{d}\) be an nc measure, and define 
    \[
        K_{\mu}(\s,\t) := \mu\big((L^{\t})^*L^{\s}\big) \text{ for words \(\s, \t \in \FF_d^+\),}
    \]
    where \(L\) is the left free shift (see \Cref{def:LeftShift}).
    Then \(K_\mu\) is a multi-Toeplitz kernel.
\end{lemma}

\begin{proof}
    For \(k,j = 1, \ldots, d\), we have the identity \(L_k^*L_j = \d_{k,j} I\) from \cite{JM22a}, so for \(\t,\s,\s' \in \FF_d^+\) where \(\t = \t_1\cdots\t_m\), \(\t_1, \ldots, \t_m \in \{1, \ldots, d\}\), we see that
    \begin{align*}
        K_{\mu}(\t\s, \t\s') & = \mu\big((L^{\t\s'})^*L^{\t\s}\big) = \mu\big( (L^{\t}L^{\s'})^* L^{\t}L^{\s}\big) = \mu\big( (L^{\s'})^* (L^{\t})^* L^{\t} L^{\s}\big) \\
        & = \mu\big( (L^{\s'})^* L_{\t_m}^* \cdots L_{\t_1}^* L_{\t_1} \cdots L_{\t_m} L^{\s}\big) = \mu\big( (L^{\s'})^* L^{\s}\big) \\
        & = K_{\mu}(\s, \s'),
    \end{align*}
    so \(K_\mu\) satisfies \eqref{CJAx1'}.

    Now, suppose that \(\s,\t \in \FF_d^+\) are such that there is no \(\a \in \FF_d^+\setminus\{\emptyset\}\) such that \(\s = \t\a\) or \(\t = \s\a\). If we write \(\s = \s_1 \cdots \s_m\) and \(\t = \t_1\cdots\t_n\) for \(\s_j,\t_j \in \{1, \ldots, d\}\) (supposing without loss of generality that \(n \leq m\)), then this means that \(\s_{j_0} \neq \t_{j_0}\) for some \(j_0 = 0, \ldots, n\): otherwise, we would have \(\s = \t \s_{n+1}\cdots\s_{m}\), contradicting the hypothesis. It follows that \(\d_{\t_{j_0},\s_{j_0}} = 0\), and we then have from \(L_k^*L_j = \d_{k,j} I\) that
    \begin{align*}
        K_{\mu}(\s,\t) & = \mu \big((L^\t)^*L^\s\big) = K_{\mu}\big( L_{\t_n}^*\cdots L_{\t_1}^* L_{\s_1} \cdots L_{\s_m}\big) \\
        & = \d_{\t_1,\s_1}\cdots \d_{\t_n,\s_n}\mu\big( L_{\s_{n+1}} \cdots L_{\s_m}\big) \\
        & = 0,
    \end{align*}
    so \(K_\mu\) satisfies \eqref{CJAx2'}.

    It remains to show that \(K_\mu\) is positive, i.e. that the square matrix \([K_{\mu}(\s,\t)]_{\s,\t \preceq \s_0}\) is positive semi-definite for any \(\s_0 \in \FF_d^+\). Fix \(\s_0 \in \FF_d^+\) and let \(n\) be its position under the (graded) lexicographical (shortlex) ordering, so that \([K_{\mu}(\s,\t)]_{\s,\t \preceq \s_0}\) is \(n\times n\), and let \(x = [x_{\emptyset}, \cdots, x_{\s_0}]^{\top} \in \CC^{n}\setminus{\{0\}}\). We see that
    \begin{align*}
        x^*[K_{\mu}(\s,\t)]_{\s,\t \preceq \s_0}x & = x^* \left[\sum_{\t \preceq \s_0} K_\mu(\s,\t) x_{\t}\right]_{\s \preceq \s_0} =\sum_{\s \preceq \s_0} \overline{x_{\s}} \sum_{\t \preceq \s_0} K_\mu(\s,\t) x_{\t}\\ 
        & = \sum_{\s,\t \preceq \s_0} K_\mu(\s,\t) \overline{x_{\s}} x_{\t} = \sum_{\s,\t \preceq \s_0} \mu\big((L^\t)^* L^\s) \overline{x_{\s}} x_{\t} \\
        & = \mu\left( \sum_{\s,\t \preceq \s_0} (\overline{x_\t}L^\t)^*(\overline{x_\s}L^\s)\right) = \mu\left(\sum_{\t \preceq \s_0} (\overline{x_\t}L^\t)^* \sum_{\s \preceq \s_0} \overline{x_\s}L^\s \right)\\
        & = \mu\left(\left(\sum_{\t \preceq \s_0} \overline{x_\t}L^\t\right)^* \left(\sum_{\s \preceq \s_0} \overline{x_\s}L^\s\right) \right),
    \end{align*}
    so that setting \(y := \sum_{\s \preceq \s_0} \overline{x_\s}L^\s \in \mathcal{A}_d\) (as a polynomial in \(L\)) we have that \(y \neq 0\) and therefore
    \[
        x^*[K_{\mu}(\s,\t)]_{\s,\t \preceq \s_0}x = \mu(y^* y) \geq 0,
    \]
    simply since \(\mu\) is positive by definition. Thus \([K_{\mu}(\s,\t)]_{\s,\t \preceq \s_0}\) is positive semi-definite and hence \(K_\mu\) is multi-Toeplitz.
\end{proof}

Classically this association is bijective, and we can show that this remains the case in the noncommutative theory. 

\begin{lemma}
    \label{lem:Correspondence2}
    Let \(K\) be a multi-Toeplitz kernel. For \(\s \in \FF_d^+\), define
    \[
        \mu_K(L^\s) := K(\s, \emptyset) \quad \text{ and } \quad \mu_K\big((L^\s)^*\big) := K(\emptyset, \s) = \overline{K(\s, \emptyset)}.
    \]
    Then \(\mu_K\) extends to a unique positive nc measure in \(\ncm{d}\), which we also denote by \(\mu_K\).
\end{lemma}

\begin{proof}
    This argument hinges on a result of \cite{JM17}, namely, Lemma 4.6, that the positive elements of the nc disk system are norm limits of sums of (Hermitian) squares; though we do not apply it directly, it is also worth mentioning the foundational \cite{Hel02} at this point as a canonical Positivestellensatz in the free noncommutative setting.

    By definition, we have \(\scrA_d = \overline{\cA_d + \cA_d^*}^{\norm{\cdot}_{\mathrm{op}}}\) and \(\cA_d = \overline{\mathrm{Alg}\{ I_{\HH^2_d}, L_1, \ldots, L_d \}}^{\norm{\cdot}_{\mathrm{op}}}\). Extension by linearity and continuity therefore yields continuous linear functionals on \(\cA_d\) and \(\cA_d^*\), and then extending by linearity and continuity a second time we obtain a continuous linear functional on \(\scrA_d\); we denote this functional by \(\mu_K\).
    
    It remains to show that \(\mu_K\) is positive, so let \(f \in \scrA_d\) be a positive element. Then by the aforementioned \cite[Lemma 4.6]{JM17}, there exists \((a_k)_{k \in \NN} \subseteq \cA_d\) such that \(f = \sum_k a_k^* a_k\), where by the infinite sum we mean the norm-limit of the partial sums. Since each \(a_k\) is in \(\cA_d\), for each \(k \in \NN\) there exist \emph{polynomials} \((p^{(k)}_j)_{j \in \NN}\) such that \(a_k = \lim_j p^{(k)}_{j}\). Each of these polynomials can be written in the usual way as
    \[
        p^{(k)}_{j} = \sum_{\s \in \FF_d^+} c^{(k)}_{j,\s} L^\s
    \]
    where the coefficients \(c^{(k)}_{j,\s}\) are eventually all zero. It follows that
    \begin{align*}
        \left(p^{(k)}_{j}\right)^* p^{(k)}_{j} & = \left( \sum_{\s \in \FF_d^+} c^{(k)}_{j,\s} L^\s \right)^* \left( \sum_{\s \in \FF_d^+} c^{(k)}_{j,\s} L^\s \right) = \left( \sum_{\s \in \FF_d^+} \overline{c}^{(k)}_{j,\s} (L^\s)^* \right)\left( \sum_{\s \in \FF_d^+} c^{(k)}_{j,\s} L^\s \right) \\
        & = \sum_{\s, \t \in \FF_d^+} \overline{c}^{(k)}_{j,\t} c^{(k)}_{j,\s} (L^\t)^* L^\s.
    \end{align*}
    For some fixed \(\s,\t \in \FF_d^+\), if we write \(\s = \s_1 \cdots \s_n\) and \(\t = \t_1 \cdots \t_m\), then \(L_k^*L_j = \d_{k,j} I\) implies that
    \[
        (L^\t)^* L^\s = \begin{cases} \d_{\s_1, \t_1} \cdots \d_{\s_n, \t_n} (L^{\t_{n+1}\cdots\t_m})^*, \quad \text{ if \(m > n\);} \\ \d_{\s_1, \t_1} \cdots \d_{\s_n, \t_n}I, \quad \text{ if \(m = n\);} \\ \d_{\s_1, \t_1} \cdots \d_{\s_m, \t_m} L^{\s_{m+1}\cdots\s_n}, \quad \text{ if \(m < n\).}\end{cases}
    \]
    Applying \(\mu_K\) yields
    \begin{align}
    \label{eqn:MomentCases}
        \begin{split}
        \mu_k((L^\s)^* L^\t) & = \begin{cases} \d_{\s_1, \t_1} \cdots \d_{\s_n, \t_n} \mu_K\big((L^{\t_{n+1}\cdots\t_m})^*\big), \quad \text{ if \(m > n\);} \\ \d_{\s_1, \t_1} \cdots \d_{\s_n, \t_n}\mu_K(I), \quad \text{ if \(m = n\);} \\ \d_{\s_1, \t_1} \cdots \d_{\s_m, \t_m} \mu_K(L^{\s_{m+1}\cdots\s_n}), \quad \text{ if \(m < n\).}\end{cases} \\
        & = \begin{cases} \d_{\s_1, \t_1} \cdots \d_{\s_n, \t_n} K(\emptyset, \t_{n+1}\cdots\t_m), \quad \text{ if \(m > n\);} \\ \d_{\s_1, \t_1} \cdots \d_{\s_n, \t_n} K(\emptyset,\emptyset), \quad \text{ if \(m = n\);} \\ \d_{\s_1, \t_1} \cdots \d_{\s_m, \t_m} K(\s_{m+1}\cdots\s_n, \emptyset), \quad \text{ if \(m < n\).}\end{cases} \\
        & = \begin{cases} \d_{\s_1, \t_1} \cdots \d_{\s_n, \t_n} K(\t_1\cdots \t_n, \t), \quad \text{ if \(m > n\);} \\ \d_{\s_1, \t_1} \cdots \d_{\s_n, \t_n}\cdot 1, \quad \text{ if \(m = n\);} \\ \d_{\s_1, \t_1} \cdots \d_{\s_m, \t_m} K(\s, \s_1\cdots\s_m), \quad \text{ if \(m < n\).}\end{cases}
        \end{split}
    \end{align}
    where the final line holds by \eqref{CJAx1'}. Notice now that this is zero if, for some \(j = 1, \ldots, \min\{m,n\}\), we have \(\d_{\s_j, \t_j} = 0\), or in other words, if there does not exist \(\a \in \FF_d^+\) such that \(\s =  \t\a\) or that \(\t = \s\a\); by \eqref{CJAx2'}, this is also the value of \(K(\s,\t)\).
    
    Alternatively, suppose such an \(\a\) does exist: by comparing successive letters of the words \(\s\) and \(\t\), we have \(\t = \s_1 \cdots \s_m\) if \(m < n\) or \(\s = \t_1 \cdots \t_n\) if \(m > n\); if \(m = n\) then \(\s = \t\), and hence \(K(\s,\t) = 1\). In this case \eqref{eqn:MomentCases} transforms into
    \begin{align*}
        \mu_k((L^\t)^* L^\s) & = \begin{cases} 1 \cdot K(\s, \t), \quad \text{ if \(m < n\);} \\ 1, \quad \text{ if \(m = n\);} \\ 1 \cdot K(\s, \t), \quad \text{ if \(m > n\)}\end{cases} \\
        & = K(\s,\t).
    \end{align*}
   Hence in every case we have
    \[
        \mu_k((L^\t)^* L^\s) = K(\s,\t).
    \]
    
    Now, denote the highest \(\s\) such that \(c^{(k)}_{j,\s} \neq 0\) by \(\s^{(k)}_{j}\). Applying linearity of \(\mu_K\), we now return to our previous calculation to see that
    \begin{align*}
        \mu_K\left(\left(p^{(k)}_{j}\right)^*p^{(k)}_{j}\right) & = \sum_{\s,\t \in \FF_d^+} \overline{c}^{(k)}_{j,\t}c^{(k)}_{j,\s} K(\s,\t)\\
        & = \mathrm{col} \left\{ c^{(k)}_{j,\s} \right\}_{\s \preceq \s^{(k)}_{j}}^* \begin{bmatrix} K(\s,\t)\end{bmatrix}_{\s,\t \preceq \s^{(k)}_{j}} \mathrm{col}\left\{ c^{(k)}_{j,\s}\right\}_{\s \preceq \s^{(k)}_{j}} \\
        & \geq 0,
    \end{align*}
    by virtue of \(K\) being positive. Finally, since \(\mu_K\) is continuous by construction, we can see that \(\mu_K(f)\) is positive:
    \begin{align*}
        \mu_K(f) & = \mu\left( \sum_k a_k^* a_k\right) = \sum_k \mu \left( (\lim_j p^{(k)}_{j})^* (\lim_j p^{(k)}_{j})\right) \\
        & = \sum_k \lim_j \mu\left( \left(p^{(k)}_{j}\right)^* p^{(k)}_{j} \right) \\
        & \geq 0
    \end{align*}
    as a limit of a sum of limits of positive objects. Therefore \(\mu_K\) is a positive functional on \(\scrA_d\), i.e. \(\mu_K \in \ncm{d}\).
\end{proof}

To summarise, we have the following correspondence between multi-Toeplitz kernels and nc measures.

\begin{theorem}
    \label{thm:Correspondence}
    Let \(K\) be a multi-Toeplitz kernel and define a functional \(\mu_K\) on \(\scrA_d\) via
    \begin{equation*}
        \mu_K(L^\s) := K(\s, \emptyset) \quad \text{ and } \quad \mu_K\big((L^\s)^*\big) := K(\emptyset, \s) = \overline{K(\s, \emptyset)}
    \end{equation*}
    and by extending by linearity and continuity. Then \(\mu_K\) is an nc measure.

    Symmetrically, let \(\mu \in \ncm{d}\) be an nc measure and define a kernel on \(\FF_d^+\) by
    \begin{equation}
        \label{eqn:Correspondence}
        K_{\mu}(\s,\t) := \mu\big((L^{\t})^*L^{\s}\big).
    \end{equation}
    Then \(K_\mu\) is a multi-Toeplitz kernel.

    Thus the formula \eqref{eqn:Correspondence} defines a bijection between the class of multi-Toeplitz kernels on \(\FF_d^+\) and \(\ncm{d}\).
\end{theorem}

\begin{proof}
    The first claim follows from \Cref{lem:Correspondence2} and the second follows from \Cref{lem:Correspondence1}.
\end{proof}

\begin{remark}
    When \(d=1\), \eqref{eqn:Correspondence} reduces to precisely the classical formula linking a measure on the unit circle and a kernel indexed by \(\NN\), i.e.
    \[
        K(n,m) = \int_0^{2\pi} e^{i(n-m)\th} \; \rmd \mu(\th),
    \]
    since the forward shift \(S\) on \(H^2(\DD)\) satisfies \(S^*S = I_{H^2(\DD)}\).
\end{remark}

\section{Verblunsky Coefficients and Orthogonal Polynomials}

Constantinescu and Johnson's paper \cite{CJ02b} associates to a positive kernel satisfying \eqref{CJAx1} and \eqref{CJAx2} a set of \emph{Schur parameters}, \((\g_{\s,\t})_{\s,\t \in \FF_d^+}\), though we prefer the nomenclature \emph{Verblunsky coefficients} following Simon \cite{Sim05}. 

This two-parameter family of complex numbers in the unit disk is constructed via a process outlined in detail in \cite{Con96} which relies on the theory of integer-indexed kernels, which we can connect to a noncommutatively-indexed kernel via the short lexicographical ordering. It should be noted that this construction is completely different to the classical derivations of Verblunsky coefficients corresponding to a measure on the unit circle (i.e. the Schur algorithm or the Szeg\H{o} recurrences); however, when \(d = 1\) they coincide as remarked upon by Constantinescu \cite{Con96}, and as such this process is an appealing generalisation to apply to our setting.

Crucially, constructing the Verblunsky coefficients \((\g_{\s,\t})_{\s,\t\in\FF_d^+}\) of a positive kernel on \(\FF_d^+\) does not at all rely on the axioms \eqref{CJAx1} and \eqref{CJAx2}, and as such, we can derive Verblunsky coefficients for a multi-Toeplitz kernel in an identical fashion. Moreover, these Verblunsky coefficients can be reduced to a one-parameter family via (2.9) in \cite{CJ02b}:
\begin{equation}
    \label{eqn:VCReduction}
    [\g_{\s,\t}]_{\lvert \s \rvert = j, \lvert \t \rvert = k} = ([\g_{\s', \t'}]_{\lvert \s' \rvert = j-1, \lvert \t' \rvert = k-1})^{\oplus d}, \quad j, k \geq 1;
\end{equation}
using this, one can build any \(\g_{\s,\t}\) from the base, single-indexed data \(\g_{\s} := \g_{\emptyset, \s}\) (and the fact that \(\g_{\t,\s} = \overline{\g_{\s,\t}}\)). Now, \eqref{eqn:VCReduction} does rely on \eqref{CJAx1} and \eqref{CJAx2}, but we can perform the same procedure using \eqref{CJAx1'} and \eqref{CJAx2'}, so that the same notion holds for multi-Toeplitz kernels.

Combined with our correspondence in Section 2, this construction allows us to define the Verblunsky coefficients of a positive nc measure with any \(d \geq 1\):

\begin{definition}
    \label{def:VerblunskyCoefficients}
    Let \(\mu\) be an nc measure and let its corresponding multi-Toeplitz kernel be \(K_\mu\). Denote by \((\g_{\s,\t})_{\s,\t\in\FF_d^+}\) the Verblunsky coefficients arising from \(K_\mu\) by considering \(K_\mu\) as a kernel on \(\NN\) via the shortlex ordering, applying \cite[Theorem 1.5.3]{Con96}, and then pulling back again via the shortlex ordering. The \emph{Verblunsky coefficients} of \(\mu\) are the family \((\g_{\s})_{\s \in \FF_d^+} \subseteq \DD\) where \(\g_{\s} := \g_{\emptyset,\s}\) is the one-parameter family determining all of \((\g_{\s,\t})_{\s,\t\in\FF_d^+}\) via \eqref{eqn:VCReduction}.
\end{definition}

In parallel to Verblunsky coefficients, \cite{CJ02b} uses a kernel \(K\) to define an inner product on nc polynomials just as one would in the commutative univariate setting: given nc polynomials \(P(Z) = \sum_\s a_\s Z^\s\) and \(Q(Z) = \sum_\s b_\s Z^\s\), define
\[
    \langle P, Q \rangle_K := \sum_{\s,\t \in \FF_d^+} K(\s,\t) a_\s \overline{b_\t}.
\]
A quick calculation, very similar to one in our proof of \Cref{lem:Correspondence1}, shows that \(\langle P,Q\rangle_{K}\) is also equal to \(\mu_K(Q^*P)\); this is in some sense the expected inner product arising from an nc measure --- indeed, a very similar construction using linear functionals appears in the related work \cite{BMV23}. Constantinescu and Johnson \cite{CJ02b} then construct the Hilbert space \(L^2(K)\) by factoring out the subspace \(\{P \in \CC\langle Z \rangle : \langle P, P \rangle_K = 0\}\) from the space \(\CC\langle Z \rangle\) of nc polynomials in \(Z_1, \ldots, Z_d\) with complex coefficients and completing with respect to the resulting norm
\[
    \norm{P}_K = \sqrt{\langle P, P \rangle_K} = \sqrt{ \mu_K(P^*P) }.
\]
If we assume that \(\begin{bmatrix}K(\s,\t)\end{bmatrix}_{\s,\t \preceq \s_0}\) is invertible for all \(\s_0 \in \FF_d^+\), then this subspace is zero, in which case we can view nc polynomials as elements of \(L^2(K)\) directly, and the set \(\{Z^\s : \s \in \FF_d^+\}\) is linearly independent in this Hilbert space. Again we emphasise that none of this relies on axioms \eqref{CJAx1} or \eqref{CJAx2}, only positivity of \(K\).

\begin{remark}
    In this remark we motivate this assumption via the univariate setting.
    
    In the case \(d = 1\), set \(K_n := \begin{bmatrix}K(k,l)\end{bmatrix}_{k,l \leq n}\) for \(n \in \NN\). If \(K_{n_0}\) is not invertible for some \(n_0 \in \NN\), then it has a nonzero vector in its kernel, say \(x\). As noted before, when \(d = 1\) axiom \eqref{CJAx1} is simply that \(K\) is Toeplitz, so we have
    \[
        K_{n+1} = \begin{bmatrix} K_n & b \\ b^* & c \end{bmatrix}
    \]
    and by assumption \(K_{n+1}\) is positive semi-definite. Smuljan's Lemma \cite{Smu59} now implies the existence of some matrix \(W\) such that \( b = K_n W\) and \(c \succeq W^* K_n W\). Positive semi-definite matrices are in particular Hermitian, so we have
    \[
        K_{n+1} \begin{bmatrix} x \\ 0 \end{bmatrix} =  \begin{bmatrix} K_n & b \\ b^* & c \end{bmatrix}\begin{bmatrix} x \\ 0 \end{bmatrix} = \begin{bmatrix} K_n x \\ b^* x \end{bmatrix} = \begin{bmatrix} K_n x \\ W^* K_n^* x \end{bmatrix} = \begin{bmatrix} K_n x \\ W^* K_n x \end{bmatrix} = \begin{bmatrix} 0 \\ 0 \end{bmatrix}
    \]
    and hence \(K_{n+1}\) is also not invertible. Indeed, we could repeat this process indefinitely; hence once \(K_n\) is not invertible for some \(n\), it turns out that \(K_m\) is not invertible for all \(m \geq n\).
    
    In the commutative theory, this implies that the measure \(\mu_K\) on the unit circle is finitely atomic \cite[Theorem 11.3]{Sch17} (and hence singular with respect to the Lebesgue measure); moreover, finitely atomic measures have finitely many Verblunsky coefficients strictly inside the unit disc \cite[Section 1.3.6]{Sim05}. This is a situation we wish to avoid: for example, such a measure has no absolutely continuous part, so statements we are interested in (such as Szeg\H{o}'s theorem itself) are built around measures where this does not happen, and as another, if some Verblunsky coefficient has \(\lvert \g_n \rvert = 1 \), then \(\lvert \g_m \rvert = 1\) for all \(m \geq n\), so statements about the square-summability of the Verblunsky coefficients become trivial and the product in \Cref{thm:CommSzego}(iv) becomes zero. The condition that an nc measure is \emph{non-trivial} in the sense we shall define shortly, then, is a noncommutative generalisation of a measure not being finitely atomic, which avoids this less desirable setting.
\end{remark}

From here on, we shall assume that our multi-Toeplitz kernels satisfy this invertibility property, and introduce the corresponding term for nc measures, motivated by the terminology of \cite{Sim05} in the classical setting.

\begin{definition}
    We say that an nc measure \(\mu \in \ncm{d}\) is \emph{non-trivial} if its corresponding multi-Toeplitz kernel \(K_\mu\) has that \(\begin{bmatrix}K_\mu(\s,\t)\end{bmatrix}_{\s,\t \preceq \s_0}\) is invertible for all \(\s_0 \in \FF_d^+\).
\end{definition}

In the classical setting, the measure \(\mu\) must often be a probability measure. We have a corresponding notion in the free noncommutative setting.

\begin{remark}
    Let \(K\) be a multi-Toeplitz kernel. By \eqref{CJAx1'}, we have that
    \[
        K(\s, \s) = K(\emptyset, \emptyset)
    \]
    for any \(\s \in \FF_d^+\), so that multi-Toeplitz kernels are constant along the ``diagonal" of \(\FF_d^+ \times \FF_d^+\). Morever, encoded in the setting of \cite{CJ02b} is the implicit assumption that their kernels are \emph{normalised}, i.e. that \(K(\emptyset, \emptyset) = 1\) and hence \(K(\s, \s) = 1\) for any \(\s \in \FF_d^+\) --- this assumption can be seen in, e.g., the example following Theorem 3.1 of \cite{CJ02b}, and arises from the application of the Schur parameter theory of \cite{Con96}. In particular, \cite[Theorem 1.5.3]{Con96} makes this assumption, and it is with this theorem that \cite{CJ02b} associates the Verblunsky coefficients to a kernel.

    For this reason, we also assume from here on that our multi-Toeplitz kernels are normalised. Indeed, in some places, such as \cite{Pop01}, this is part of the definition of a multi-Toeplitz kernel. Taking into account our \Cref{thm:Correspondence}, this has a natural analogue in the setting of nc measures: if \(\mu_K\) is the nc measure corresponding to \(K\), then
    \[
        K(\emptyset, \emptyset) = \mu_K\left(\big(L^{\emptyset}\big)^* L^{\emptyset}\right) = \mu_K(I_{\HH^2_d}).
    \]
    Recall that when \(d = 1\), \(\mu_K\) is the integration functional on \(C(\TT)\) with respect to some measure \(\tilde{\mu}_K\), so that \(\mu(1) = \int_0^{2\pi} \, \rmd \tilde{\mu}_K\). Hence \(K\) being normalised is equivalent to \(\tilde{\mu}_K\) being a probability measure. 
\end{remark}

This motivates the next definition.

\begin{definition}
    Let \(\mu \in \ncm{d}\) be an nc measure. We say that \(\mu\) is a \emph{probability nc measure} if \(\mu(I_{\HH^2_d}) = 1\).
\end{definition}

As \(K\) is implicitly always normalised in \cite{CJ02b}, we shall assert this condition in much of the remainder of the paper.

\begin{definition}
    Let \(\mu\) be a non-trivial nc measure. The space \(L^2(\mu)\) is the Hilbert space gained by completing the space \(\CC\langle Z \rangle\) of nc polynomials with respect to the norm
    \[
        \norm{P}_\mu := \sqrt{\mu(P^*P)},
    \]
    and has inner product
    \[
        \langle P, Q \rangle_\mu := \mu(Q^* P).
    \]
\end{definition}

Next, the authors of \cite{CJ02b} go on to apply Gram-Schmidt to the (linearly independent) set of monomials \(\{Z^\s : \s \in \FF_d^+\}\) with respect to \(\langle \cdot, \cdot \rangle_K\), ordered by the short lexicographical ordering. They obtain nc polynomials \(\{\p_{\s} : \s \in \FF_d^+\}\) which are orthonormal with respect to \(\langle\cdot,\cdot\rangle_K\); much of the remainder of the paper is devoted to exploring consequences of these nc orthonormal polynomials and their relationship to the Verblunsky coefficients of \(K\). In particular, they derive recurrence relations for the orthonormal polynomials in Theorem 3.2 that compare well with the classical Szeg\H{o} recurrences, in the process acquiring a set of ``reverse polynomials", though these are defined via these recurrence relations rather than through an explicit formula, as one has in the commutative case.

An exploration of commutative Szeg\H{o} theory, as in e.g. \cite{Sim05}, quickly finds a heavy reliance both on orthonormal polynomials and on monic orthogonal polynomials, which are themselves the orthonormal polynomials rescaled to be monic. However, \cite{CJ02b} focuses on other topics, omitting any consideration of noncommutative monic orthogonal polynomials. The work already present in \cite{CJ02b} does, however, lead quickly to several observations if we \emph{do} define these polynomials.

\begin{definition}
    Given a non-trivial probability nc measure \(\mu\), we perform Gram-Schmidt in \(L^2(\mu)\) on the set \(\{Z^{\s} : \s \in \FF_d^+\}\) with respect to the shortlex ordering. Omitting normalisation associates to \(\mu\) by definition a family of \emph{orthogonal monic polynomials}, which we denote \((\P_\s)_{\s \in \FF_d^+}\), while normalising gives a family of \emph{orthonormal polynomials}, which we denote \((\p_\s)_{\s\in\FF_d^+}\).
\end{definition}

\begin{definition}
    Let \(\mu \in \ncm{d}\) be a non-trivial probability nc measure and let \((\p_\s)_{\s \in \FF_d^+}\) be the orthonormal polynomials associated to \(\mu\). For \(\s \in \FF_d^+\), the coefficients of \(\p_\s\) shall be denoted by \((a_{\s,\t})_{\emptyset \preceq \t \preceq \s}\) so that we may write
    \[
        \p_\s = \sum_{\emptyset \preceq \t \preceq \s} a_{\s,\t} Z^{\t}.
    \]
\end{definition}

\begin{remark}
    By considering the normalisation at each stage of Gram-Schmidt, the monic orthogonal polynomials are related to the orthonormal polynomials via \(\P_\s = \frac{1}{a_{\s,\s}}\p_\s\).
\end{remark} 

This observation allows us to note a number of facts about these polynomials. Firstly, since \(\p_\s\) has \(\mu\)-norm 1, the \(\mu\)-norm of \(\P_\s\) is \(\frac{1}{a_{\s,\s}}\). We then have by \cite[Theorem 3.1(2)]{CJ02b} that (for \(\s = \s_1 \cdots \s_k\))
\[
    \norm{\P_\s}^2_\mu = \frac{1}{a_{\s,\s}^2} = \prod_{1 \leq j \leq k} (1 - \lvert \g_{\s_j \cdots \s_k} \rvert^2).
\]
An equality just like this appears in \cite{Sim05} for the norm of the monic orthogonal polynomials. However, a subtle point to note is that in one variable, this product is in fact over \emph{all} Verblunsky coefficients up to a certain index, whereas for us, this product is over only a particular subset of the Verblunsky coefficients preceding \(\s\).

Perhaps more importantly, we can use the recurrence relations for \((\p_\s)_{\s \in \FF_d^+}\) of \cite[Theorem 3.2]{CJ02b}, i.e.
\begin{equation}
    \label{eqn:recNormal}
    \p_{k\s} = \frac{1}{\sqrt{1 - \lvert \g_{k\s}} \rvert^2}(Z_k\p_\s - \g_{k\s} \p_{k\s - 1}^{\#})
\end{equation}
and
\begin{equation}
    \label{eqn:recReverseNormal}
    \p_{k\s}^{\#} = \frac{1}{\sqrt{1 - \lvert \g_{k\s}\rvert^2}}(-\overline{\g_{k\s}} Z_k \p_\s + \p_{k\s-1}^{\#}),
\end{equation}
to derive corresponding ones for the monic polynomials.

\begin{remark}
    These recurrences reduce to the classical relations one finds in \cite{Sim05} in the case \(d = 1\) due to the fact that \(k\s - 1 = \s\) in this case. However, there is a subtlety when \(d > 1\) not present in the \(d = 1\) setting: as \(k\s - 1\) no longer equals \(\s\) in general, on the right hand side of both \eqref{eqn:recNormal} and \eqref{eqn:recReverseNormal} we note that the index of the orthonormal polynomial differs from the index of the reverse polynomial, which may in some sense be surprising from the univariate theory.
\end{remark}

\begin{remark}
    It should be noted that we have some freedom to make a definition here: as the reverse polynomials of \cite{CJ02b} are defined via recurrence relations rather than a general direct formula as we have in the classical theory, the noncommutative reverse polynomials of the monic orthogonal polynomials are not immediately determined by those of the orthonormal polynomials. The most natural thing to do would be to follow the behaviour observed classically, that is, if \(\p_n(z) = \sum_j a_j z^j\) then \(\P^*_n(z) = \frac{1}{a_n}\p^*_n(z)\). Making a definition analogous to this phenomenon, we can obtain recurrence relations for these orthogonal polynomials from those of the orthonormal ones.
\end{remark}

\begin{definition}
    \label{def:reversePolynomials}
    Let \(\mu\) be a non-trivial probability nc measure with monic orthogonal polynomials \((\P_\s)_{\s \in \FF_d^+}\) and orthonormal polynomials \((\p_\s)_{\s \in \FF_d^+}\) and let the leading coefficient of \(\p_\s\) be \(a_{\s,\s}\). The \emph{reverse polynomials} of the orthonormal polynomials are the polynomials \((\p_\s^{\#})_{\s\in\FF_d^+}\) appearing in \eqref{eqn:recNormal} and \eqref{eqn:recReverseNormal}, and for \(\s \in \FF_d^+\), we define the \emph{reverse polynomial} of \(\P_\s\) to be \(\P_\s^{\#} := \frac{1}{a_{\s,\s}}\p_\s^\#\). 
\end{definition}

\begin{theorem}
    \label{thm:NCRecurrences}
    Let \(\mu\) be a non-trivial probability nc measure with associated monic orthogonal polynomials \((\P_\s)_{\s \in \FF_d^+}\). We have the following analogue of the Szeg\H{o} recurrence relations:
    \begin{equation}
        \label{eqn:recMonic}
        \P_{k\s} = Z_k \P_\s - \g_{k\s}\frac{a_{k\s -1, k\s -1}}{a_{\s,\s}} \P^\#_{k\s - 1}
    \end{equation}
    and
    \begin{equation}
        \label{eqn:recReverseMonic}
        \P^\#_{k\s} = -\overline{\g_{k\s}} Z_k \P_\s + \frac{a_{k\s -1, k\s -1}}{a_{\s,\s}}  \P^\#_{k\s - 1}.
    \end{equation}
\end{theorem}
\begin{proof}
    This follows readily from substitution of our definitions into the recurrence relations \eqref{eqn:recNormal} and \eqref{eqn:recReverseNormal} and some algebraic observations. Recall from an equation in the proof of \cite[Theorem 3.1]{CJ02b} that for \(\s = \s_1\cdots\s_N \in \FF_d^+, k = 1, \ldots, d\) we have
    \[
        a_{\s,\s} = \prod_{1 \leq j \leq N} (1 - \lvert \g_{\s_j \cdots \s_N} \rvert^2)^{-\frac{1}{2}};
    \]
    for \(\t = k\s\) we can write \(\t = \t_1 \t_2 \cdots \t_{N+1} = k \s_1 \cdots \s_N\), so that
    \begin{align*}
        a_{k\s,k\s} = a_{\t,\t} & = \prod_{1 \leq j \leq N+1} (1 - \lvert \g_{\t_j\cdots \t_{N+1}} \rvert^2)^{-\frac{1}{2}}\\
        & = (1 - \lvert \g_{k\s} \rvert^2)^{-\frac{1}{2}} \prod_{2 \leq j \leq N+1} (1 - \lvert \g_{\t_j \cdots \t_{N+1}} \rvert^2)^{-\frac{1}{2}} \\
        & = \frac{1}{d_{k\s}} \prod_{1 \leq j \leq N} (1 - \lvert \g_{\s_j \cdots \s_{N}} \rvert^2)^{-\frac{1}{2}}\\
        & = \frac{a_{\s,\s}}{d_{k\s}}.
    \end{align*}
    Now, substituting \(\P_\s = a_{\s,\s} \p_\s\) and \(\P_\s^{\#} = a_{\s,\s} \p_\s^{\#}\) into the Szeg\H{o} recurrences for the orthonormal polynomials, we obtain from \eqref{eqn:recNormal} that
    \[
        d_{k\s}a_{k\s,k\s} \P_{k\s} = a_{\s,\s} Z_k \P_s - \g_{k\s} a_{k\s-1,k\s-1} \P_{k\s-1}^{\#}
    \]
    and with the above computation it follows that
    \[
        \P_{k\s} = Z_k \P_\s - \g_{k\s}\frac{a_{k\s -1, k\s -1}}{a_{\s,\s}} \P^\#_{k\s - 1}.
    \]
    Similarly, \eqref{eqn:recReverseNormal} becomes
    \[
        d_{k\s} a_{k\s,k\s} \P_{k\s}^{\#} = -\overline{\g_{k\s}} a_{\s,\s} Z_k \P_\s + a_{k\s-1,k\s-1}\P_{k\s-1}^{\#}
    \]
    which reveals the identity
    \[
        \P^\#_{k\s} = -\overline{\g_{k\s}} Z_k \P_\s + \frac{a_{k\s -1, k\s -1}}{a_{\s,\s}}  \P^\#_{k\s - 1}.
    \]
\end{proof}

\begin{remark} 
    The factor of \(\frac{a_{k\s -1, k\s -1}}{a_{\s,\s}}\) appearing in \eqref{eqn:recMonic} and \eqref{eqn:recReverseMonic} is trivially always 1 when \(d = 1\), and in that case these compare precisely to the Szeg\H{o} recurrence relations (see \cite{Sim05} with a relabelling of the Verblunsky coefficients).
\end{remark}

\begin{remark}
    The structure of the recurrence relations \eqref{eqn:recNormal} and \eqref{eqn:recReverseNormal} (and similarly \eqref{eqn:recMonic} and \eqref{eqn:recReverseMonic}) are twofold. Firstly, for \(\s \in \FF_d^+\) we obtain a relationship between \(\p_{\s}\) and \(\p_{\s+1}\), i.e. between immediately shortlex-successive orthonormal polynomials, and secondly, we see a relationship between \(\p_\s\) and \(\p_{k\s}\) for each \(k = 1, \ldots, d\). In this way these recurrences encode a relationship that intertwines the monoidal and order structures on \(\FF_d^+\).
\end{remark}

What we now have is a way to associate Verblunsky coefficients and families of both monic orthogonal and orthonormal polynomials to any non-trivial probability nc measure, up to and including recurrence relations for both families of polynomials. As these polynomials are used extensively in the classical theory, this puts us in a strong position to explore analogues of those classical results in \(d\) noncommuting variables.

Finally, in the commutative theory, Verblunsky's theorem (or Favard's theorem on the unit circle) provides the converse of this statement: given a candidate sequence \((\eta_n)_{n=0}^{\infty}\) in the unit disc, there exists a measure \(\mu\) on the unit circle whose Verblunsky coefficients are \(\g_n = \eta_n, n \in \NN\), and when all of these lie strictly inside the disc, \(\mu\) is unique. In \cite[Theorem 5.1]{CJ02b} we see the equivalent result for our setting, which we restate here in the language of nc measures.

\begin{theorem}
    \label{thm:NCVerblunsky}
    Let \((\g_\s)_{\s \in \FF_d^+}\) be a family of complex numbers with \(\g_\emptyset = 0\) and \(\lvert \g_\s \rvert < 1\) for \(\s \in \FF_d^+ \setminus \{\emptyset\}\). Then there exists a unique non-trivial probability nc measure \(\mu\) such that the polynomials \((\p_\s)_{\s\in\FF_d^+}\) defined by the recurrences \eqref{eqn:recNormal} and \eqref{eqn:recReverseNormal} with \(\p_\emptyset = \p_\emptyset^{\#} = 1\) are orthogonal with respect to the inner product induced by \(\mu\).
\end{theorem}

This combined with \Cref{thm:NCRecurrences} provides a full correspondence between noncommutatively-indexed families of complex numbers inside the unit disc on the one side and non-trivial probability nc measures on the other, directly analogous to the commutative theory.

\section{The NC Christoffel Function}

In the univariate setting, a central role is played by the Christoffel-Darboux kernel, and accordingly the Christoffel function associated to a measure. In this section we use the theory of orthogonal polynomials to introduce a noncommutative analogue that will play a similarly central role to our setting.

\begin{definition}
    \label{def:ncCDK}
    Let \(\mu\) be a non-trivial probability nc measure with orthonormal polynomials \((\p_\s)_{\s \in \FF_d^+}\). The \(n^{\mathrm{th}}\) \emph{nc Christoffel-Darboux kernel} associated to \(\mu\) is the map sending \(k\times k\) matrices \(A,B\) to
    \begin{equation}
        \label{eqn:CDKernel}
        \k_{\mu, n}(A,B) := \sum_{\s = \emptyset}^{\s(n)} \p_\s(A)\p_\s(B)^* = I_k + \sum_{\emptyset \prec \s \preceq \s(n)}  \p_\s(A)\p_\s(B)^*
    \end{equation}
    and the \(n^{\mathrm{th}}\) \emph{nc Christoffel approximate} is
    \[
        \L_{n}(A; \mu) := \k_{\mu, n}(A,A)^{-1}.
    \]
\end{definition}

We use the same Christoffel-Darboux kernels as \cite{CJ02b}, whence we see an alternate formula analogous to the Christoffel-Darboux formula \cite[Theorem 4.2]{CJ02b}.

\begin{remark} 
    A very similar construction appears in \cite{BMV23}, but our setting requires some differences to this construction as \cite{BMV23} imposes that the generators \(X_1, \ldots, X_d\) be self-adjoint, which is not a property we impose for polynomials in the disk system.
\end{remark}

\begin{remark}
    \label{rem:BMV}
    Classically, the Christoffel approximates \(\l_n(\cdot; \rmd \mu)\) are minima of a certain integral quantity over polynomials of degree up to \(n\) which map the origin to 1, and it is a theorem that these minima are equal to the inverse of a sum of the form \eqref{eqn:CDKernel}. We next prove an analogue of this result in the noncommutative setting, though for convenience we proceed in the opposite direction, showing that our \(\L_{n}(\cdot; \mu)\) defined in terms of orthogonal polynomials is in fact equal to a certain minimum. We proceed by mimicking the proof of \cite[Theorem 3.4]{BMV23}, which is a very similar result, but our setting has some minor differences that require slightly different computations --- for example, while the involution \(\star\) of \cite{BMV23} appears on the right in \(PP^{\star}\), the structure of the \(d\)-shift on \(\HH^2_d\) encourages us to consider quantities with an adjoint on the left in \(Q^*Q\); this allows us to bring the machinery of Jury and Martin to bear. Nevertheless, the spirit of the proof in \cite{BMV23} carries through.
\end{remark}

In \cite[Section 3]{BMV23}, two tensor products between complex algebras are described. Given algebras \(\cV, \cW\) over \(\CC\), one has a choice of two product operations on the algebraic tensor product \(\cV \otimes \cW\): either \(\cV \otimes \cW\), the algebra with product
\begin{equation}
    \label{eqn:FirstTensorProduct}
    \left(\sum_i v_i \otimes w_i\right)\left(\sum_i \tilde{v_i} \otimes \tilde{w_i}\right) = \sum_{i,j} v_i \tilde{v_j} \otimes w_i \tilde{w_j}
\end{equation}
or \(\cV \otimes \cW^{\mathrm{op}}\), the algebra with product
\begin{equation}
    \label{eqn:SecondTensorProduct}
    \left(\sum_i v_i \otimes w_i\right)\left(\sum_i \tilde{v_i} \otimes \tilde{w_i}\right) = \sum_{i,j} v_i \tilde{v_j} \otimes \tilde{w_j} w_i.
\end{equation}

Belinschi, Magron and Vinnikov consider \((\CC\langle X \rangle)^{k \times k} \cong \CC^{k \times k} \otimes \CC\langle X \rangle\) with the \emph{opposite product} \eqref{eqn:SecondTensorProduct}, as this is of interest in the study of operator spaces and has applications down the line in their work. This interest arises as one can identify \(P(Z) = \sum_{\s} c_\s \otimes Z^{\s} \in (\CC\langle X \rangle)^{k\times k}\) as taking values in linear maps on matrices: for a tuple of \(k\times k\) matrices \(A\), \(P(A) = \sum_\s c_\s \otimes A^{\s}\) sends \(C \in \CC^{k\times k}\) to \(\sum_\s c_\s C A^{\s}\), and the opposite product on the tensor product corresponds to the usual composition of these operators. See the in-depth discussion at the beginning of \cite[Section 3.2]{BMV23} for further details.

On the other hand, we shall choose the tensor product algebra structure \((\CC\langle Z \rangle)^{k\times k} \cong \CC \langle Z \rangle \otimes \CC^{k\times k}\) with the product \eqref{eqn:FirstTensorProduct}, as this structure will allow us to have the adjoint on the left as discussed in \Cref{rem:BMV}. In this algebra, one may still consider \(P(A) = \sum_\s c_\s \otimes A^{\s}\) as a linear map on \(\CC^{k\times k}\), however, this embedding of \(\CC^{k\times k} \otimes \CC^{k\times k} \hookrightarrow B(\CC^{k\times k}, \CC^{k\times k})\) is no longer a homomorphism as the product \eqref{eqn:FirstTensorProduct} no longer corresponds to the composition of the relevant linear operators. This will suffice for our setting as we need not evaluate a product of two matrix-valued nc polynomials as a linear map.

Using this choice of product, one may work with a polynomial expressed in terms of the orthonormal polynomials \((\p_\s)_{\s\in\FF_d+}\) directly, rather than in terms of their adjoints --- this is beneficial to us as, unlike the setting of \cite{BMV23} where \(P^{\star}\) is a polynomial when \(P\) is, the adjoint of \(\p_\s\) is no longer itself a polynomial.

\begin{remark}
    In the sequel, we in fact only need to evaluate the linear operator \(P(A)\), where \(P \in (\CC\langle Z \rangle)^{k\times k}\) and \(A \in \CC^{k\times k}\), on the \(k\times k\) identity matrix \(I_k\). For this reason, we shall suppress this secondary dependence, so that for us \(P(A)\) shall mean not the linear operator, but rather the matrix-valued quantity \(P(A)[I_k]\) in the notation of \cite{BMV23}.
\end{remark}

\begin{lemma}
    \label{lem:ChristoffelFunction}
    Let \(\mu\) be a non-trivial probability nc measure with Christoffel approximates \((\L_n(\cdot; \mu))_{n\in\NN}\). For \(n \in \NN\) and \(A \in (\CC^{k\times k})^d\), we have
    \[
        \L_{n}(A; \mu) = \min\{(\mu \otimes \mathrm{Id}_k)(Q^*Q) \; : \; Q \in (\CC\langle Z \rangle_{\s(n)})^{k\times k}, \; Q(A) = I_k\}.
    \]
\end{lemma}
\begin{proof}
    For \(\s \in \FF_d^+\) and \(n \in \NN\), let \(c_\s(A) = \p_\s(A)^*\L_{n}(A; \mu)\) and define the \(k\times k\) matrix of nc polynomials \(Q_n\) by
    \[
        Q_n := \sum_{\s = \emptyset}^{\s(n)} \p_\s \otimes c_\s(A).
    \]
    Note that, unlike the corresponding step in \cite{BMV23}, this is expressed in terms of the basis \((\p_\s)_{\s\in\FF_d^+}\), rather than in terms of \((\p_\s^*)_{\s\in\FF_d^+}\) in analogy to \cite[Theorem 3.4]{BMV23}: this is because while \((P^{\star}_\s)_{\s\in\FF_d^+}\) is also a basis of \(\CC\langle X \rangle\) in their setting, \(\p_\s^*\) is not even a polynomial in \(Z\) for us.

    Now,
    \begin{align*}
        Q_n(A) & = \sum_{\s = \emptyset}^{\s(n)} \p_\s(A) c_\s(A) = \sum_{\s = \emptyset}^{\s(n)} \p_\s(A) \p_\s(A)^*\L_{n}(A; \mu) \\
        & = \L_{n}(A; \mu)^{-1} \L_{n}(A; \mu) \\
        & = I_k,
    \end{align*}
    so that \(Q_n\) is admissible for the minimum of interest. Next, we see from our choice of the tensor structure \eqref{eqn:FirstTensorProduct} that
    \[
        Q_n^* Q_n = \sum_{\s,\t = \emptyset}^{\s(n)} \p_\s^*\p_\t \otimes c_\s^*(A)c_\t(A)
    \]
    and subsequently we calculate that
    \begin{align*}
        (\mu \otimes \mathrm{Id}_k)(Q_n^* Q_n) & = \sum_{\s,\t = \emptyset}^{\s(n)} \mu(\p_\s^* \p_\t)c_\s(A)^*c_\t(A) \\
        & = \sum_{\s,\t = \emptyset}^{\s(n)} \langle \p_\t, \p_\s \rangle_\mu c_\s(A)^*c_\t(A) \\
        & = \sum_{\s,\t = \emptyset}^{\s(n)} \d_{\s,\t} c_\s(A)^* c_\t(A) \\
        & = \sum_{\s = \emptyset}^{\s(n)} c_\s(A)^* c_\s(A) \\
        & = \sum_{\s = \emptyset}^{\s(n)} \big( \L_{n}(A; \mu)^* \p_\s(A)\big) \big( \p_\s(A)^*\L_{n}(A; \mu)\big) \\
        & = \L_{n}(A; \mu) \left(\sum_{\s = \emptyset}^{\s(n)} \p_\s(A)\p_\s(A)^*\right) \L_{n}(A; \mu)^* \\
        & = I_k \cdot \L_{n}(A; \mu)^* \\
        & = \left(\left(\sum_{\s = \emptyset}^{\s(n)} \p_\s(A)\p_\s(A)^*\right)^{-1}\right)^* \\
        & = \left(\left(\sum_{\s = \emptyset}^{\s(n)} \p_\s(A)\p_\s(A)^*\right)^{*}\right)^{-1} \\
        & = \left(\sum_{\s = \emptyset}^{\s(n)} \p_\s(A)\p_\s(A)^*\right)^{-1} \\
        & = \L_{n}(A; \mu).
    \end{align*}
    Hence \(\L_{n}(A; \mu)\) is in the set we want to minimise over.

    Now let \(P = \sum_{\s = \emptyset}^{\s(n)} \p_\s \otimes a_\s \in (\CC\langle Z \rangle_{\s(n)})^{k\times k}\) be arbitrary and suppose that \(P(A) = I_k\), i.e. that \(\sum_{\s = \emptyset}^{\s(n)} \p_\s(A) a_\s = I_k\). In terms of the natural Hilbert module structure over \(\CC\langle Z \rangle\), this means precisely that \(\langle(\p_\s(A))_{\s = \emptyset}^{\s(n)}, (a^*_\s)_{\s = \emptyset}^{\s(n)}\rangle = I_k\). Considering the matrix
    \[
        X = \begin{bmatrix}\p_{\emptyset}(A) & \cdots & \p_{\s(n)}(A) \\ a_{\emptyset}^* & \cdots & a_{\s(n)}^* \end{bmatrix}
    \]
    and applying \cite[Theorem 14.10]{Pau02}, we obtain positivity of the matrix 
    \[
        Y = \begin{bmatrix} \sum_{\s = \emptyset}^{\s(n)} \langle\p_\s(A), \p_\s(A)\rangle & \sum_{\s = \emptyset}^{\s(n)} \langle \p_\s(A), a_\s^*\rangle \\ \sum_{\s = \emptyset}^{\s(n)} \langle a_\s^*, \p_\s(A) \rangle & \sum_{\s = \emptyset}^{\s(n)} \langle a_\s^*, a_\s^* \rangle \end{bmatrix} = \begin{bmatrix} \sum_{\s = \emptyset}^{\s(n)} \p_\s(A) \p_\s(A)^* & \sum_{\s = \emptyset}^{\s(n)} \p_\s(A) a_\s \\ \sum_{\s = \emptyset}^{\s(n)} a_\s^* \p_\s(A)^* & \sum_{\s = \emptyset}^{\s(n)} a_\s^* a_\s \end{bmatrix},
    \]
    and on the other hand, \(P(A) = I_k\) implies that
    \[
        Y = \begin{bmatrix}\L_{n}(A; \mu)^{-1} & I_k \\ I_k & \sum_{\s = \emptyset}^{\s(n)} a_\s^* a_\s \end{bmatrix}.
    \]
    We note that
    \[
        \L_n(A; \mu)^{-1} = \k_{\mu, n}(A,A) = I_k + \sum_{\emptyset \prec \s \preceq \s(n)}  \p_\s(A)\p_\s(A)^*,
    \]
    which is evidently positive definite. Via the Schur complement (e.g. see \cite[Remark 2.4.6]{BW11}), then, \(Y\) is positive semi-definite if and only if
    \[
        \sum_{\s = \emptyset}^{\s(n)} a_\s^* a_\s - I_k^* \left(\L_{n}(A; \mu)^{-1}\right)^{-1}I_k \succeq 0,
    \]
    or equivalently,
    \[
        \L_n(A; \mu) \preceq \sum_{\s = \emptyset}^{\s(n)} a_\s^* a_\s.
    \]
    
    Finally, note that
    \[
        (\mu \otimes \mathrm{Id}_k)(P^*P) = \sum_{\s = \emptyset}^{\s(n)} a_\s^* a_\s.
    \]
    As a result we now see that, for any \(P \in (\CC\langle Z \rangle_{\s(n)})^{k\times k}\), it holds that
    \[
        \L_{n}(A; \mu) \preceq (\mu \otimes \mathrm{Id}_k)(P^*P),
    \]
    and so \(\L_{n}(A; \mu)\) is the minimum claimed.    
\end{proof}

\begin{remark}
    \label{rem:Minimiser}
    The proof of the preceding lemma is constructive, that is, not only do we show that \(\L_n(A; \mu)\) is the minimum claimed, but for each \(n \in \NN\) and \(A \in \CC^{k\times k}\) we in fact construct an nc polynomial \(Q_n\) such that \((\mu \otimes \mathrm{Id}_k)(Q_n^* Q_n) = \L_n(A;\mu)\). In addition to the lemma, the existence of this explicit minimiser shall prove useful to us later. We emphasise that this sequence is dependent upon the choice of input \(A \in \CC^{k\times k}\).
\end{remark}

We can use this to obtain an nc Christoffel function, analogous to the function \(\l_{\infty}(\cdot ; \rmd \mu)\) of \cite{Sim05}.

\begin{theorem}
    \label{thm:CFExists}
    The pointwise limit of the Christoffel approximates,
    \[
        \lim_{n\to\infty} \L_{n}(A; \mu),
    \]
    exists for all \(A \in \CC^{k\times k}\), \(k \in \NN\), and is given by
    \[
        \lim_{n\to\infty} \L_{n}(A; \mu) = \inf\{((\mu \otimes \mathrm{Id}_k))(Q^*Q) \; : \; Q \in (\CC\langle Z \rangle)^{k\times k}, \; Q(A) = I_k\}.
    \]
\end{theorem}
\begin{proof}
    For \(m,n \in \NN\) with \(m > n\), we have the inclusion \((\CC\langle Z \rangle_{\s(n)})^{k\times k} \subset (\CC\langle Z \rangle_{\s(m)})^{k\times k}\). Hence the minimiser \(Q_n\) from \Cref{rem:Minimiser} is also an element of the set minimised by \(Q_m\), so that
    \[
        \L_{n}(A; \mu) = (\mu \otimes \mathrm{Id}_k)(Q_n^*Q_n) \succeq \L_{m}(A; \mu) = (\mu \otimes \mathrm{Id}_k)(Q_m^*Q_m).
    \]
    This gives us a decreasing sequence of matrices \(Z_n := (\mu \otimes \mathrm{Id}_k)(Q_n^*Q_n)\) which are all themselves positive semi-definite, i.e.
    \[
        Z_1 \succeq Z_2 \succeq \cdots \succeq Z_{n} \succeq Z_{n+1} \succeq \cdots \succeq 0_{k\times k}.
    \]

    Each of these matrices has an associated positive semi-definite sesquilinear form given by \(q_n(x,y) := y^*Z_nx\); for each \(x \in \CC^{k}\), we can hence associate a sequence in \(\RR\) via \(a_n(x) := x^*Z_nx\). Each of these sequences is decreasing: since \(Z_n - Z_{n+1} \succeq 0_{k\times k}\), we have
    \[
        a_n(x) - a_{n+1}(x) = x^*(Z_n - Z_{n+1})x \geq 0
    \]
    and hence \(a_n(x) \geq a_{n+1}(x)\).
    
    Again by positive semi-definiteness, these sequences are bounded below by zero, so each one converges in \(\CC\). Denote by \(a(x)\) the limit \(\lim_{n\to\infty} a_n(x)\). We claim that there exists a unique positive semi-definite \(Z \in \CC^{k\times k}\) such that \(a(x) = x^*Zx\).
    
    Consider the polarisation identity for \(j \in \NN\),
    \[
        q_n(x,y) := \frac{1}{4}\sum_{k=0}^3 a_n(x+i^ky).
    \]
    Since the limits \(a(x+i^k y)\) all exist (from the above), taking the limit, we obtain
    \[
        q(x,y) := \frac{1}{4}\sum_{k=0}^3 a(x+i^ky) = \lim_{n\to\infty} q_n(x,y).
    \]
    The limit of sesquilinear forms must be itself a sesquilinear form, and so \(q\) is sesquilinear; it follows that there exists some matrix \(Z \in \CC^{k\times k}\) such that \(q(x,y) = y^* Z x\). Furthermore, we have that
    \[
        q_n(x,x) - q(x,x) = x^* Z_n x - x^*Zx = x^*(Z_n - Z)x;
    \]
    since the sequences \(a_n(x+i^kx)\) are decreasing, so too are the sequences \(q_n(x,x)\), and it follows that \(q_n(x,x) - q(x,x) \geq 0\) for all \(x \in \CC^k\). From this we can see that \( x^*(Z_n - Z)x \geq 0\) for all \(x \in \CC^k\), i.e. \(Z_n \succeq Z\). On the other hand, since each \(a_n(x) \geq 0\), taking limits we have \(a(x) \geq 0\) and hence \(q(x,x) \geq 0\), so that \(Z\) is positive semi-definite, and the sequence \((Z_d)_{d \in \NN}\) converges to a positive semi-definite matrix \(Z\).
    
    With this \(Z\) in hand, we then have for any \(x \in \CC^k\) and \(n \in \NN\), that
    \[
        x(Z_n - Z)x^* = xZ_nx^* - xZx^* = a_n(x) - a(x) \geq 0;
    \]
    this shows that \(xZ_nx^* \geq xZx^*\) and therefore \(Z_n \succeq Z\) for all \(n \in \NN\). Hence the limit \(Z = \lim_{n\to\infty} \L_{n}(A; \mu)\) exists. This allows us to write
    \begin{align*}
        \lim_{n\to\infty} \L_{n}(A; \mu) & = \lim_{n\to\infty} \min\{(\mu \otimes \mathrm{Id}_k)(Q^*Q) : Q \in (\CC\langle Z \rangle_{\s(n)})^{k\times k}, \; Q(A) = I_k\} \\
        & = \inf\{(\mu \otimes \mathrm{Id}_k)(Q^*Q) : Q \in (\CC\langle Z \rangle)^{k\times k}, \; Q(A) = I_k\}.
    \end{align*}
    To elaborate on this last equality: \(Z\) may not be in the set
    \[
        \{(\mu \otimes \mathrm{Id}_k)(Q^*Q) : Q \in (\CC\langle Z \rangle_{\s(n)})^{k\times k}, \; Q(A) = I_k\}
    \]
    for any particular \(n \in \NN\), but nevertheless, for any \(Y\) in this set, there exists \(j \in \NN\) such that \(Z_j \preceq Y\), and as we have just seen, we then have \(Z \preceq Z_j\). Hence \(Z\) provides the desired notion of infimum.
\end{proof}

This existence result allows us to make the following definition.

\begin{definition}
    \label{def:ncChristoffelFunction}
    Let \(\mu\) be a non-trivial probability nc measure. The \emph{nc Christoffel function} associated to \(\mu\), denoted \(\L_{\infty}(\cdot; \mu)\), is the pointwise limit of the nc Christoffel approximates, and is given by
    \[
        \L_{\infty}(A; \mu) := \lim_{n\to\infty} \L_{n}(A; \mu) = \inf\{(\mu \otimes \mathrm{Id}_k)(Q^*Q) : Q \in (\CC\langle Z \rangle)^{k\times k}, \; Q(A) = I_k\}
    \]
    for any \(A \in \CC^{k\times k}\) and \(k\in \NN\).
\end{definition}

The final result of this section relies on the Szeg\H{o}--Verblunsky theorem (see \eqref{eqn:SzegoVerblunsky}) to see that singular nc measures in the sense of Jury and Martin \cite{JM22b} have zero nc Christoffel function at zero. Let \(0^{(d)}_{1\times1}\) be the \(d-\)tuple of scalar zeros \((0, \ldots, 0) \in \BB^d_1\).

\begin{lemma}
    \label{lem:ZeroEntropy}
    If \(\mu\) is a singular nc measure, then \(\L_{\infty}(0_{1\times1}^{(d)}; \mu) = 0\).
\end{lemma}
\begin{proof}
    
    Since \(\mu\) is positive, the infimum form of \(\L_{\infty}(0_{1\times 1}^{(d)}; \mu)\) shows that \(\L_{\infty}(0_{1\times 1}^{(d)}; \mu)\) is non-negative, and it follows that
    \begin{align*}
        0 & \leq \L_{\infty}(0_{1\times 1}^{(d)}; \mu) \leq \inf\{\mu(Q^*Q) : Q \in \CC\langle Z_1\rangle, \; Q(0) = 1\},
    \end{align*}
    using the canonical embedding of \(\CC\langle Z_1\rangle\) into \(\CC\langle Z \rangle\). We thus restrict \(\mu\) to the case \(d = 1\), and as \cite[Corollary 8.5]{JM22b} demonstrates, in this setting the nc Lebesgue decomposition coincides with the classical one. When we restrict \(\mu\) to the copy of \(\cA_1\) generated by \(Z_1\), then, we get a functional \(\mu \vert_{\cA_1} : \cA_1 \to \CC\); since \(\cA_1\) is isomorphic to the space \(C(\TT)\) of continuous functions on the circle \(\TT\), we can further use the Riesz--Markov--Kakutani representation theorem \cite[Theorem 2.17]{Rud87} to see that this corresponds to integration with respect to a finite, regular Borel measure on the circle. Moreover, due to the nc and classical Lebesgue decompositions coinciding, this measure will be singular; we shall denote this singular measure by \(\tilde{\mu}\).
    
    This provides the computation
    \begin{align*}
        \inf\{\mu(Q^*Q) : Q \in \CC\langle Z_1\rangle, \; Q(0) = 1\} & = \inf\left\{\int_0^{2\pi} Q(e^{i\theta})^*Q(e^{i\theta}) \; \rmd \tilde{\mu}(\theta) : Q \in \CC\langle Z_1\rangle, Q(0) = 1\right\}\\
        & = \inf\left\{\int_0^{2\pi} \lvert Q(e^{i\theta}) \rvert^2 \; \rmd \tilde{\mu}(\theta) : Q \in \CC[z], Q(0) = 1\right\} \\
        & = \l_{\infty}(0 ; \rmd \tilde{\mu}),
    \end{align*}
    where \(\l_{\infty}(\cdot; \rmd \tilde{\mu})\) is the Christoffel function associated to \(\tilde{\mu}\) in the sense of \cite{Sim05}. Now, since \(\tilde{\mu}\) is singular, its (classical) Lebesgue decomposition \(\tilde{\mu} = \tilde{\mu}_{ac} + \tilde{\mu}_s\) satisfies \(\tilde{\mu}_{ac} = w(\th) \frac{\rmd \th}{2\pi} = 0\), i.e. \(w = 0\). Applying the Szeg\H{o}--Verblunsky theorem in the classical setting to \(\tilde{\mu}\) demonstrates the equality
    \[
        \l_{\infty}(0 ; \rmd \tilde{\mu}) = \exp\left(\int_0^{2\pi} \log(w(\th)) \; \frac{\rmd \th}{2\pi}\right) = \exp\left(\int_0^{2\pi} \log(0) \; \frac{\rmd \th}{2\pi}\right) = 0,
    \]
    so that 
    \[
        0 \leq \L_{\infty}(0_{1\times 1}^{(d)}; \mu) \leq 0,
    \]
    i.e. \(\L_{\infty}(0_{1\times 1}^{(d)}; \mu) = 0\).
\end{proof}

In the commutative univariate setting, the Christoffel function at zero depends only upon the absolutely continuous part \(\mu_{\ac}\) of the measure \(\mu\). One can see this e.g. via \Cref{thm:CommSzego}: \(\mu\) and \(\mu_{\ac}\) have the same Radon--Nikodym derivative \(w \in L^1(\TT)\), so that
\[
    \l_{\infty}(0; \rmd\mu) = \exp\left(\int_0^{2\pi} \log\big(w(\th)\big) \, \frac{\rmd\th}{2\pi}\right) = \l_{\infty}(0; \rmd\mu_{\ac}).
\]
In the free noncommutative setting, we can use \Cref{lem:ZeroEntropy} to establish an analogous fact.

\begin{corollary}
    \label{cor:CFDependsOnAC}
    Let \(\mu = \mu_{\ac} + \mu_{\rms}\) be the nc Lebesgue decomposition of a non-trivial probability nc measure \(\mu\). Then
    \[
        \L_{\infty}(0^{(d)}_{1\times1}; \mu) = \L_{\infty}(0^{(d)}_{1\times1}; \mu_{\ac}).
    \]
\end{corollary}
\begin{proof}
    Recall that for positive real-valued functions \(F,G\) on a set \(X\), we have that \(\inf_X\{F + G\} = \inf_X\{F\} + \inf_X\{G\}\). Since nc measures are by definition positive and \(Q^*Q\) is a positive element of \(\scrA_d\) for any nc polynomial \(Q \in \CC\langle Z \rangle\), it follows that
        \begin{align*}
             \L_{\infty}(0_{1 \times 1}^{(d)}; \mu) & = \inf\{\mu_{ac}(Q^*Q) + \mu_{s}(Q^*Q) : Q \in \CC\langle Z \rangle, Q(0) = 1\} \\
             & =  \inf\{\mu_{ac}(Q^*Q) : Q \in \CC\langle Z \rangle, Q(0) = 1\} +  \inf\{\mu_{s}(Q^*Q) : Q \in \CC\langle Z \rangle, Q(0) = 1\};
        \end{align*}
        since \(\mu_s\) is singular, however, the second term is zero by \Cref{lem:ZeroEntropy}, so that
        \[
            \L_{\infty}(0_{1 \times 1}^{(d)}; \mu) = \inf\{\mu_{ac}(Q^*Q) : Q \in \CC\langle X \langle, Q(0) = 1\} = \L_{\infty}(0^{(d)}_{1\times1}; \mu_{\ac}).
        \]
\end{proof}

\begin{remark}
    It is easy to check that \(\L_{\infty}(\cdot; \mu)\) respects direct sums (though not similarities, and hence is not an nc function). For this reason, it is sufficient to evaluate on \(0_{1 \times 1}^{(d)}\) in the above results: any tuple of larger zero matrices, say \(0_{k\times k}^{(d)} \in (\CC^{k\times k})^d\), decomposes as a direct sum of \(k\) copies of \(0_{1 \times 1}^{(d)}\) and hence
    \[
        \L_{\infty}(0_{k\times k}^{(d)}; \mu) = \bigoplus_{j=1}^{k} \L_{\infty}(0_{1\times 1}^{(d)}; \mu) = \L_{\infty}(0_{1\times 1}^{(d)}; \mu) I_k.
    \]
    In other words, \(\L_{\infty}(0_{k\times k}^{(d)}; \mu)\) is determined solely by the value of \(\L_{\infty}(0_{k\times k}^{(d)}; \mu)\).
\end{remark}

\section{A Noncommutative Szeg\H{o}-Type Theorem}

This section will unite the work done in the previous three sections to prove the main result of the paper, \Cref{thm:NCSzego}.

First, we recall some notation. For \(d \geq 2\), let \(\mu \in (\cA_d^\dagger)_+\) be a non-trivial probability nc measure with absolutely continuous part \(\mu_{ac}\) and nc Radon--Nikodym derivative \(T\), Verblunsky coefficients \((\g_\s)_{\s \in \FF_d^+}\), monic orthogonal polynomials \((\P_\s)_{\s \in \FF_d^+}\), orthonormal polynomials \((\p_\s)_{\s \in \FF_d^+}\), multi-Toeplitz determinants \((D_\s)_{\s \in \FF_d^+}\), nc Christoffel approximates 
\[
    \L_{n}(A; \mu) = \left(\sum_{\t = \emptyset}^{\s(n)} \p_{\t}(A)\p_{\t}(A)^*\right)^{-1}
\]
and nc Christoffel function \(\L_{\infty}(\cdot, \mu)\). Let the leading coefficient of \(\p_\s\) be \(a_{\s,\s}\). We denote by \((\p^{\#}_\s)_{\s \in \FF_d^+}\) the polynomials satisfying the nc Szeg\H{o} recurrence relations with \((\p_\s)_{\s \in \FF_d^+}\) in \cite[Theorem 3.2]{CJ02b}, and similarly to the commutative theory we define \(\P^{\#}_\s := \frac{1}{a_{\s,\s}}\p^{\#}_\s\). Finally, by \(\s(n)\) we mean the last word of length \(n\) under the shortlex ordering of \(\FF_d^+\), i.e. \(\s(n) = d\cdots d\) with \(n\) copies of \(d\), with \(\s(0) = \emptyset\), and by \(\s - 1\) and \(\s + 1\) we mean the immediate predecessor and successor of \(\s\) in this ordering.

Of relevance shortly will be the following well-known lemma (see, e.g., the remark following \cite[Lemma 3.3.5]{Kat04}) and its consequence. 

\begin{lemma}
    \label{lem:ProductSumConvergence}
    Let \((a_n)_{n=0}^{\infty}\) be a sequence of real numbers with \(0 \leq a_n \leq 1\) for all \(n \in \NN\). Then 
    \[
        \prod_{n=0}^{\infty} (1 - a_n)
    \]
    converges (and is nonzero) if and only if
    \[
        \sum_{n=0}^{\infty} a_n
    \]
    converges. In particular, if \(\sum_{n=0}^{\infty} a_n = +\infty\) then \(\prod_{n=0}^{\infty} (1 - a_n) = 0\).
\end{lemma} 

\begin{corollary}
    \label{cor:SquareSummableVCs}
    Let \(\mu \in \ncm{d}\) be a non-trivial probability nc measure with Verblunsky coefficients \((\g_\s)_{\s \in \FF_d^+}\). The Verblunsky coefficients are square-summable, i.e. \(\sum_{\s \in \FF_d^+} \lvert \g_\s \rvert^2 < +\infty\), if and only if \(\prod_{\s \in \FF_d^+} (1 - \lvert \g_\s \rvert^2)\) converges and is nonzero.
\end{corollary}
\begin{proof}
    Using the shortlex ordering on \(\FF_d^+\), let \(a_n\) be the \(n^{\rm{th}}\) element of \((\lvert \g_\s \rvert^2)_{\s \in \FF_d^+}\). The corollary is immediate from \Cref{lem:ProductSumConvergence} with this choice of \(a_n\).
\end{proof}

We are now in a position to formulate and prove our main result; let \(0^{(d)}_{1\times1}\) be the \(d-\)tuple of scalar zeros. Recall the sequence of minimisers \(Q_n\) from \Cref{rem:Minimiser} with the choice \(A = 0_{1\times1}^{(d)}\).

\begin{theorem}
    \label{thm:NCSzego}
    Let \(\mu \in \ncm{d}\) be a non-trivial probability nc measure. The following quantities are equal.
    \begin{enumerate}
        \item[\rm{(i)}] \(\k : = \lim_{n \to \infty} \norm{\P_{\s(n)}}^2\);
        \item[\rm{(ii)}] \(\lim_{n\to\infty} a_{\s(n),\s(n)}^{-2}\);
        \item[\rm{(iii)}] \(\prod_{n=0}^\infty (1 - \lvert \g_{\s(n)} \rvert^2)\);
        \item[\rm{(iv)}] \(\lim_{n\to\infty} \frac{D_{\s(n)}}{D_{\s(n) - 1}}\).
    \end{enumerate}
    Moreover, the following quantities are equal:
    \begin{enumerate}
        \item[\rm{(v)}] \(\tilde{\k} := \prod_{\s \in \FF_d^+} (1 - \lvert \g_\s \rvert^2)\);
        \item[\rm{(vi)}] \(\L_{\infty}(0^{(d)}_{1 \times 1}; \mu)\);
        \item[\rm{(vi)'}] \(\L_{\infty}(0^{(d)}_{1 \times 1}; \mu_{\ac})\);
        \item[\rm{(vii)}] \(\lim_{n\to\infty} \lvert \p_{\s(n)}^{\#}(0^{(d)}_{1 \times 1}) \rvert^{-2}\);
        \item[\rm{(viii)}] \(\lim_{n\to\infty} \left(\sum_{\s=\emptyset}^{\s(n)} \lvert \p_\s(0^{(d)}_{1 \times 1}) \rvert^2\right)^{-1}\);
        \item[\rm{(ix)}] \(\lim_{n \to \infty}\langle TQ_n, Q_n \rangle \) for some \(Q_n \in \CC\langle X \langle_{\s(n)}, n \in \NN\);
        \item[\rm{(x)}] \(\EE(\tilde{\mu}) = \int_0^\infty t \; \rmd \tilde{\mu}(t)\) for some positive Borel measure \(\tilde{\mu}\) on \([0,\infty)\),
    \end{enumerate}
    and by comparing (iii) and (v) we see that
    \[
        \k = \tilde{\k} \cdot \prod_{\s \in \FF_d^+\setminus\{\s(n) : n \in \NN\}} (1 - \lvert \g_\s \rvert^2).
    \]
\end{theorem}
\begin{proof}
    The proof is broken into steps.
    \begin{enumerate}[align=left]
        \item[\textbf{Step 1:} \textit{Proof that} (i) = (ii):] This item is straightforward. Since the polynomials \(\P_\s\) are the (monic) results of Gram-Schmidt, and the polynomials \(\p_\s\) are normalisations of those, equating coefficients of \(Z^\s\) we can use
        \[
            \p_\s := \frac{\P_\s}{\norm{\P_\s}}
        \]
        to see that
        \[
            a_{\s,\s} = \frac{1}{\norm{\P_\s}},
        \]
        or in other words, that \(\norm{\P_\s} = a_{\s,\s}^{-1}\) for all \(\s \in \FF_d^+\), and the equality follows.

        \bigskip

        \item[\textbf{Step 2:} \textit{Proof that} (ii) = (iii) = (iv):] In the proof of \cite[Theorem 3.1]{CJ02b} appears the equation (for arbitrary \(\s = \s_1\cdots\s_m \in \FF_d^+\))
        \[
            \frac{1}{a_{\s,\s}^2} = \frac{D_\s}{D_{\s -1}} = \prod_{1 \leq j \leq m} (1 - \lvert \g_{\s_j\cdots\s_m}\rvert^2);
        \]
        taking specifically \(\s = \s(m)\), we have \(\s_j = d\) for \(j = 1, \ldots, m\), so the product on the right-hand side becomes (writing \(j\) copies of \(d\) as \(d^j\))
        \[
            \prod_{1 \leq j \leq m} (1 - \lvert \g_{d^j}\rvert^2) = \prod_{n = 1}^m (1 - \lvert \g_{\s(n)} \rvert^2).
        \]
        The equation from \cite{CJ02b} under these choices is then
        \[
            \frac{1}{a_{\s(m),\s(m)}^2} = \frac{D_{\s(m)}}{D_{\s(m) - 1}} = \prod_{n = 1}^m (1 - \lvert \g_{\s(n)} \rvert^2),
        \]
        and recalling from \cite{CJ02b} that \(\g_{\s(0)} = \g_{\emptyset} = 0\), we can take this product from \(n=0\) without changing the value. Taking the limits as \(m \to \infty\) shows the desired equalities.

        \bigskip
                
        \item[\textbf{Step 3:} \textit{Proof that} (vii) = (v):] Recall the recurrence relation \eqref{eqn:recReverseNormal} for \(\s \in \FF_d^+, k \in \{1, \ldots, d\}\):
        \[
            \p_{k\s}^{\#} = \frac{1}{d_{k\s}}(-\overline{\g_{k\s}}Z_k\p_\s + \p^{\#}_{k\s - 1}),
        \]
        and also that \(\p_\emptyset = \p^{\#}_{\emptyset} = 1\) are constant polynomials. First, notice that 
        \[
            \p^{\#}_{1} = \frac{1}{d_1}( - \overline{\g_1}Z_1\p_{\emptyset} + \p^{\#}_{\emptyset}) = \frac{1 - \overline{g_1}Z_1}{d_1}
        \]
        which gives
        \[
            \p^{\#}_{1}(0^{(d)}_{1 \times 1}) = \frac{1}{d_1} = \frac{1}{d_\emptyset} \frac{1}{d_1},
        \]
        recalling again that \(d_{\emptyset} = 1 - \lvert 0 \rvert^2 = 1\).
        
        With this as our base case, suppose now for some \(\s \in \FF_d^+\) that \(\p^{\#}_\s(0) = \prod_{\t = \emptyset}^{\s}d_\t^{-1}\). The same recurrence relation as above shows us that
        \[
            \p^{\#}_{\s + 1}(0^{(d)}_{1 \times 1}) = \frac{1}{d_{\s+1}} \p^{\#}_{\s}(0^{(d)}_{1 \times 1}) = \frac{1}{d_{\s+1}}\prod_{\t = \emptyset}^{\s}d_\t^{-1} = \prod_{\t = \emptyset}^{\s + 1}d_\t^{-1};
        \]
        by induction on the place of \(\s\) in the shortlex ordering of \(\FF_d^+\) then, 
        \[
            \p^{\#}_\s(0^{(d)}_{1 \times 1}) = \prod_{\t = \emptyset}^{\s}d_\t^{-1} = \prod_{\t = \emptyset}^{\s} \left(\sqrt{1 - \lvert \g_\t \rvert^2}\right)^{-1}
        \]
        for all \(\s \in \FF_d^+\). It is now quick to see that
        \[
            \lvert \p^{\#}_\s(0^{(d)}_{1 \times 1}) \rvert^{-2} = \left \lvert \prod_{\t = \emptyset}^{\s} \left(\sqrt{1 - \lvert \g_\t \rvert^2}\right)^{-1} \right\rvert^{-2} = \prod_{\t = \emptyset}^{\s} (1 - \lvert \g_\t \rvert^2);
        \]
        since this holds for all \(\s \in \FF_d^+\), in particular it holds for \(\s(n)\) for all \(n \in \NN\), and then taking the limit as \(n \to \infty\) completes the proof.

        \bigskip

        \item[\textbf{Step 4:} \textit{Proof that} (viii) = (v):] First, notice that \eqref{eqn:recNormal} implies that
        \[
            \p_\s(0^{(d)}_{1 \times 1}) = \frac{1}{d_\s}(-\g_\s \p^{\#}_{\s - 1}(0^{(d)}_{1\times1})) = \frac{-\g_\s}{d_\s} \prod_{\t = \emptyset}^{\s -1} \frac{1}{d_\t} = -\g_\s \prod_{\t = \emptyset}^{\s} \frac{1}{d_\t},
        \]
        where the second equality is from our proof of ``(vii) = (v)". Next, as a base case, consider that
        \[
            \sum_{\s=\emptyset}^{\s(0)} \lvert \p_\s(0^{(d)}_{1 \times 1}) \rvert^2 = \lvert \p_{\emptyset}(0^{(d)}_{1\times1}) \rvert^2 = 1 = \frac{1}{d_{\emptyset}^2};
        \]
        with this in mind, suppose for some \(\s \in \FF_d^+\) that
        \[
            \sum_{\t=\emptyset}^{\s} \lvert \p_\t(0^{(d)}_{1 \times 1}) \rvert^2 = \prod_{\t = \emptyset}^{\s} \frac{1}{d_\t}^2.
        \]
        Then direct computation shows that
        \begin{align*}
            \sum_{\t=\emptyset}^{\s + 1} \lvert \p_\t(0^{(d)}_{1 \times 1}) \rvert^2 & = \prod_{\t = \emptyset}^{\s} \frac{1}{d_\t}^2 + \lvert \p_{\s + 1}(0^{(d)}_{1 \times 1}) \rvert^2 = \prod_{\t = \emptyset}^{\s} \frac{1}{d_\t}^2 + \lvert\g_{\s + 1}\rvert^2 \prod_{\t = \emptyset}^{\s + 1} \frac{1}{d_\t}^2 \\
            & = \frac{d_{\s + 1}^2 + \lvert \g_{\s + 1} \rvert^2}{\prod_{\t = \emptyset}^{\s + 1} {d_\t}^2} = \frac{1 - \lvert \g_{\s+1} \rvert^2 + \lvert \g_{\s+1} \rvert^2}{\prod_{\t = \emptyset}^{\s + 1} {d_\t}^2} \\
            & = \prod_{\t = \emptyset}^{\s + 1} \frac{1}{d_\t}^2,
        \end{align*}
        so by induction we have
        \[
            \sum_{\t=\emptyset}^{\s} \lvert \p_\t(0^{(d)}_{1 \times 1}) \rvert^2 = \prod_{\t = \emptyset}^{\s} \frac{1}{d_\t}^2 =  \prod_{\t = \emptyset}^{\s} (1 - \lvert \g_\s \rvert^2)^{-1}
        \]
        for all \(\s \in \FF_d^+\). Hence in particular
        \[
            \left(\sum_{\t=\emptyset}^{\s(n)} \lvert \p_\t(0^{(d)}_{1 \times 1}) \rvert^2\right)^{-1} = \prod_{\t = \emptyset}^{\s(n)} (1 - \lvert \g_\s \rvert^2)
        \]
        and taking the limit as \(n \to \infty\) completes this proof.

        \bigskip

        \item[\textbf{Step 5:} \textit{Proof that} (vi) = (viii):] This item also follows almost immediately. Taking \(k = 1\), we have from our definition of \(\L_{n}(\cdot; \mu)\) that
        \[
            \L_{n}(0^{(d)}_{1 \times 1}; \mu) = \left(\sum_{\s = \emptyset}^{\s(n)} \p_\s(0^{(d)}_{1 \times 1}) \overline{\p_\s(0^{(d)}_{1 \times 1})}\right)^{-1} = \left(\sum_{\t=\emptyset}^{\s(n)} \lvert \p_\t(0^{(d)}_{1 \times 1}) \rvert^2\right)^{-1},
        \]
        and once more taking \(n \to \infty\) proves the equality.

        \bigskip

        \item[\textbf{Step 6:} \textit{Proof that} (vi) = (vi)':] This was proved in \Cref{cor:CFDependsOnAC}. 

        \bigskip
        
        \item[\textbf{Step 7:} \textit{Proof that} (vi)' = (ix):] Recall that \(T\) is the nc Radon--Nikodym derivative of \(\mu\) \cite{JM22a}, that is, the (generally unbounded) positive, closed, densely-defined linear operator \(\HH_d^2 \to \HH_d^2\) such that
        \[
            \mu_{ac}(a^*b) = \langle \sqrt{T}b, \sqrt{T}a \rangle_{\HH^2_d}.
        \]
        Since \(T\) is closed, densely-defined, and positive, \(T\) has a unique positive square root \(\sqrt{T}\) with \(\cD(\sqrt{T}) = \cD(T)\), so for \(Q \in \CC\langle Z \rangle\), we have \(\mu_{ac}(Q^*Q) = \langle \sqrt{T}Q, \sqrt{T}Q \rangle_{\HH^2_d} = \langle TQ, Q \rangle_{\HH^2_d}\). Moreover, recall the minimiser \(Q_n\) for \(\min\{\mu(Q^*Q) : Q \in \CC\langle Z \rangle_{\s(n)}, Q(0) = 1\}\) from \Cref{rem:Minimiser}. Combine this with the above to see that
        \begin{align*}
            \inf\{\mu_{ac}(Q^*Q) : Q \in \CC\langle X \langle, Q(0) = 1\} & = \lim_{n\to\infty} \min\{\mu_{ac}(Q^*Q) : Q \in \CC\langle Z \rangle_{\s(n)}, Q(0) = 1\} \\
            & = \lim_{n\to\infty} \mu_{ac}(Q_n^*Q_n) \\
            & = \lim_{n\to\infty} \langle TQ_n, Q_n \rangle_{\HH^2_d}
        \end{align*}
        as claimed.

        \bigskip
        
        \item[\textbf{Step 8:} \textit{Proof that} (ix) = (x):] First, suppose that the Verblunsky coefficients of \(\mu\) are not square-summable, i.e. that \(\sum_{\s \in \FF_d^+} \lvert \g_\s \rvert^2 = +\infty\). Then by \Cref{cor:SquareSummableVCs}, we have that the product \(\prod_{\s\in\FF_d^+}(1-\lvert \g_\s \rvert^2) = 0\). Taking \(\tilde{\mu} = 0\) to be the zero measure we immediately obtain a candidate measure:
        \[
            \prod_{\s \in \FF_d^+} (1 - \lvert \g_\s \rvert^2) = 0 = \int_0^{\infty} t \; \rmd \tilde{\mu} = \EE(\tilde{\mu}),
        \]
        and we are done.
        
        On the other hand, suppose that the Verblunsky coefficients are square-summable, i.e. that the sum \(\sum_{\s \in \FF_d^+} \lvert \g_\s \rvert^2\) is finite. Once again, the operator \(T\) of \cite{JM22a} is densely-defined, closed, and positive, so we can apply the spectral theorem for an unbounded self-adjoint operator (see e.g. \cite[Theorem 5.8]{Sch12}) to obtain a projection-valued spectral measure \(E\) on the spectrum \(\s(T)\), such that \(E(\s(T)) = I_{\HH_d^2}\). This measure will satisfy
        \[
            T = \int_{\RR^+} t \; \rmd E(t)
        \]
        and --- to us more importantly ---
        \[
            \ip{Tf}{g}_{\HH^2_d} = \int_{\RR^+} t \; \rmd \ip{E(t)f}{g}_{\HH^2_d},
        \]
        where \(\ip{E(\cdot)f}{g}_{\HH^2_d}\) is a positive scalar measure on \(\s(T)\) determined by \(f,g \in \HH^2_d\). Following on from the previous item, then, we see that
        \[
            \L_{\infty}(0_{1 \times 1}^{(d)}; \mu) = \lim_{n\to\infty} \int_{\RR^+} t \; \rmd \ip{E(t)Q_n}{Q_n}_{\HH^2_d}
        \]
        where \((\ip{E(\cdot)Q_n}{Q_n}_{\HH^2_d})_{n\in\NN}\) is a sequence of positive measures on the positive half-line. For ease of notation we shall write \(\nu_n\) in place of \(\ip{E(\cdot)Q_n}{Q_n}_{\HH^2_d}\).

        For \(j = 0,1\) and \(n \in \NN\), define the moments
        \[
            a_j^{(n)} := \int_0^{\infty} t^{j} \; \rmd \nu_n,
        \]
        so that \(\lim_n a_1^{(n)} = \lim_n \mathbb{E}(\nu_n) = \L_{\infty}(0; \mu)\) and
        \[
            a_0^{(n)} = \int_0^{\infty} \rmd \nu_n = \langle E(\s(T))Q_n, Q_n \rangle_{\HH_d^2} = \langle Q_n, Q_n \rangle_{\HH_d^2} = \norm{Q_n}^2_{\HH_d^2}.
        \]
        
        We show that this cannot have a limit of zero. Suppose to the contrary that \(\lim_n \norm{Q_n}^2_{\HH_d^2} = 0\), and notice that \(\norm{\sqrt{T}Q_n}_{\HH_d^2} = \langle \sqrt{T}Q_n, \sqrt{T}Q_n\rangle_{\HH_d^2} = \langle TQ_n, Q_n \rangle_{\HH_d^2}\) converges, so \((\sqrt{T}Q_n)_n\) converges in \(\HH_d^2\). Then \cite[Proposition 1.4]{Sch12} says that 
        \[
            0 = \sqrt{T}(0) = \sqrt{T}\left(\lim_n Q_{n}\right) = \lim_n \sqrt{T}Q_{n}
        \]
        and, iterating this calculation, that
        \[
            0 = \sqrt{T}(0) = \sqrt{T}\left(\lim_n \sqrt{T}Q_{n}\right) = \lim_n TQ_{n}.
        \]
        Thus \(\L_{\infty}(0; \mu) = \lim_n \langle TQ_{n}, Q_{n} \rangle_{\HH_d^2} = 0\), but by \Cref{cor:SquareSummableVCs}, when \(\sum_{\s} \lvert \g_{\s} \rvert^2 < +\infty\) this is not the case. By contradiction we thus demonstrate that \(\lim_n a_0^{(n)} \neq 0\) and hence that \(\inf_n \{a_0^{(n)}\} > 0\).

        By \cite[Theorem 1.8]{Sch12}, the closure of \(T\) is \(T\) itself; Equation (1.6) of the same source then means that we can decompose \(\HH_d^2\) as
        \[
            \HH_d^2 = \overline{\mathrm{Ran}(T)} \oplus \ker(T).
        \]
        The operator \(T\) is then unitarily equivalent to the operator matrix \(\begin{bmatrix}\tilde{T} & 0 \\ 0 & 0\end{bmatrix}\), where \(\tilde{T} : \overline{\rm{Ran}(T)} \to \overline{\rm{Ran}(T)}\) is positive, closed, and densely-defined.

        Notice that, if \(Q_n \in \ker(T)\) for some \(n \in \NN\), then for \(m \geq n\), the set we minimise over for \(\L_{m}(0; \mu)\) also contains \(Q_n\), and hence \(\L_{m}(0; \mu) = 0\) for all \(m \geq n\) and so too does \(\L_{\infty}(0; \mu)\). As this would violate \Cref{cor:SquareSummableVCs}, we must have that \(Q_n \not\in \ker( T)\) for all \(n \in \NN\). Considering our decomposition of \(\HH_d^2\), this means that \(Q_n \in \overline{\rm{Ran}(T)}\) and its component in \(\ker(T)\) is zero, i.e. \(Q_n\) is written as the vector \(\begin{bmatrix}Q_n \\ 0\end{bmatrix}\).

        Now,  \(\ker(\tilde{T}) = \{0\}\) so there exists some \(\a > 0\) such that \(\s(\tilde{T}) \subseteq [\a, \infty)\), and the spectral theorem applied to \(\tilde{T}\) provides us with a spectral measure \(\tilde{E}\) on \([\a,\infty)\). The integral form of \(\tilde{T}\) then implies
        \begin{align*}
            \langle TQ_n, Q_n \rangle_{\HH_d^2} & = \left\langle \begin{bmatrix}\tilde{T} & 0 \\ 0 & 0\end{bmatrix} \begin{bmatrix}Q_n \\ 0\end{bmatrix}, \begin{bmatrix}Q_n \\ 0\end{bmatrix} \right\rangle_{\HH_d^2} \\
            & = \langle \tilde{T}Q_n,Q_n \rangle_{\overline{\rm{Ran}(T)}} \\
            & = \int_\a^{\infty} t \; \rmd \langle \tilde{E}(t) Q_n, Q_n \rangle_{\overline{\rm{Ran}(T)}} \\
            & \geq \int_\a^\infty \a \; \rmd \langle \tilde{E}(t) Q_n, Q_n \rangle_{\overline{\rm{Ran}(T)}} \\
            & = \a \int_\a^\infty 1 \; \rmd \langle \tilde{E}(t) Q_n, Q_n \rangle_{\overline{\rm{Ran}(T)}} \\
            & = \a \; \rmd \langle \tilde{E}(\s(\tilde{T}))Q_n, Q_n \rangle_{\overline{\rm{Ran}(T)}} \\
            & = \a \langle Q_n, Q_n \rangle_{\overline{\rm{Ran}(T)}} \\
            & = \a \langle  \begin{bmatrix}Q_n \\ 0\end{bmatrix}, \begin{bmatrix}Q_n \\ 0\end{bmatrix} \rangle_{\HH_d^2} \\
            & = \a \langle Q_n, Q_n \rangle_{\HH_d^2} \\
            & = \a a_0^{(n)},
        \end{align*}
        i.e. 
        \[
            a_0^{(n)} \leq \frac{1}{\a} \langle TQ_n, Q_n \rangle_{\HH_d^2}.
        \]

        We have already seen that the right hand side converges, and so the left hand side cannot diverge to \(+\infty\), i.e. \(\sup_n a_0^{(n)} < \infty\). Since \((a_1^{(n)})_{n\in \NN}\) is convergent, \((a_1^{(n)})_{n\in \NN}\), too, is bounded. 

        We thus have for all \(n \in \NN\) that
        \[
            \frac{a_1^{(n)}}{a_0^{(n)}} \in \left[\frac{\inf_n a^{(n)}_1}{\sup_n a_0^{(n)}}, \frac{\sup_n a^{(n)}_1}{\inf_n a_0^{(n)}}\right],
        \]
        i.e. the sequence \((\frac{a_1^{(n)}}{a_0^{(n)}})_{n \in \NN}\) lies inside a compact set.

        Define now a sequence of point measures
        \[
            \mu_n := a_0^{(n)}\d_{\frac{a_1^{(n)}}{a_0^{(n)}}}.
        \]
        By construction, the zeroth and first moments of \(\mu_n\) will be, respectively, \(a_0^{(n)}\) and \(a_1^{(n)}\), so that \(\mu_n\) is a representing measure for the truncated moment sequence \((a_0^{(n)}, a_1^{(n)})\).

        Recall that the space of continuous functions on \(\RR^+\) which vanish at infinity, \(C_0(\RR^+)\), is a separable Banach space, by e.g. the same argument as \cite{Sto01} with \(F = \RR^+\). The Riesz--Markov--Kakutani representation theorem (see e.g. \cite[Theorem 2.17]{Rud87}) provides a bijection which allows us to obtain from each \(\mu_n\) a functional on this space, \(\hat{\mu}_n\). Explicitly, this functional is given by (for \(f \in C_0(\RR^+)\))
        \[
            \hat{\mu}_n(f) = \int_{\RR^+} f \; \rmd \mu_n = a_0^{(n)}f\left(\frac{a_1^{(n)}}{a_0^{(n)}}\right).
        \]
        We can compute the norms of these functionals:
        \[
            \norm{\hat{\mu}_n} = \sup_{f \in B_1(C_0(\RR^+))} \lvert \hat{\mu}_n(f) \rvert = \sup_{f \in B_1(C_0(\RR^+))} a_0^{(n)} \left\lvert f\left(\frac{a_1^{(n)}}{a_0^{(n)}}\right)\right\rvert,
        \]
        and if \(f \in  B_1(C_0(\RR^+))\), then
        \[
            \left\lvert f\left(\frac{a_1^{(n)}}{a_0^{(n)}}\right)\right\rvert \leq \sup_{t \in \RR^+} \lvert f(t) \rvert = \norm{f}_{\infty} \leq 1,
        \]
        showing \(\norm{\hat{\mu}_n} \leq a_0^{(n)}\). The reverse inequality is clear, e.g. via constructing a piecewise-linear function in \(C_0(\RR^+)\) with \(f\left(\frac{a_1^{(n)}}{a_0^{(n)}}\right) = 1\), and so in fact
        \[
            \norm{\hat{\mu}_n} = a_0^{(n)}.
        \]
        We have already shown that this quantity is bounded above; if we let some upper bound be \(M > 0\), it will follow that for all \(n \in \NN\), \(\hat{\mu}_n \in MB_1(C_0(\RR^+)^*)\).

        Since \(C_0(\RR^+)\) is a Banach space, \cite[V.3.1]{Con90} gives us that \(MB_1(C_0(\RR^+)^*)\) is weak-* compact. Since \(C_0(\RR^+)\) is, additionally, separable, \cite[V.5.1]{Con90} tells us that \(MB_1(C_0(\RR^+)^*)\) is weak-* metrisable. It follows that there exists a subsequence of functionals \((\hat{\mu}_{n_j})_{j=0}^{\infty}\) that is weak-* convergent to some functional, say \(\psi\), in  \(MB_1(C_0(\RR^+)^*)\), and moreover, this functional is positive: if \(f \geq 0\) then \(\psi(f) = \lim_{j \to \infty} \hat{\mu}_{n_j}(f) \geq 0 \) as a limit of non-negative elements. One final time, then, by the Riesz--Markov--Kakutani representation theorem we obtain a positive finite regular Borel measure \(\tilde{\mu}\) on \(\RR^+\) such that \(\psi(f) = \int_{\RR^+} f \; \rmd \tilde{\mu}\).

        Let \(f(t) = t\). The expectation of this measure is given by
        \begin{multline*}
            \EE(\tilde{\mu}) = \int_{\RR^+} t \; \rmd \tilde{\mu}(t) = \psi(f) = \lim_j \hat{\mu}_{n_j}(f) = \lim_j \EE(\mu_j) \\ = \lim_j a_1^{(n_j)} = \lim_j \langle TQ_{n_j}, Q_{n_j} \rangle_{\HH_d^2} = \lim_n \langle TQ_{n}, Q_{n} \rangle_{\HH_d^2} = \prod_{\s\in\FF+d^+}(1 - \lvert \g_\s \rvert^2),
        \end{multline*}
        which completes the proof and our list.
    \end{enumerate}

    Finally, the relation between \(\k\) and \(\tilde{\k}\) is clear from observing the forms of (iii) and (v).

\end{proof}

\begin{remark}
    We remark that many of the items on the list of \Cref{thm:NCSzego} are equal, not only in the limit, but also pointwise --- for example, we see in the proof of ``(i) = (ii)" that \(\norm{\P_\s} = a_{\s,\s}^{-1}\) for all \(\s \in \FF_d^+\); however, we state the theorem in the form above both to emphasise the parallel with \Cref{thm:CommSzego} as well as to allow for items such as (vi) and (x) which are not expressed here as limits of other discussed quantities.
\end{remark}

\begin{remark}
    \label{rem:Popescu06}
    There is a clearly-defined relationship between the quantity \(\L_{\infty}(0_{1\times1}^{(d)}; \mu)\), i.e. the second list of \Cref{thm:NCSzego}, and the \emph{prediction entropy} of a multi-Toeplitz operator on the full Fock space discussed in \cite{Pop06}. Recall that the full Fock space on \(d\) generators \(F^2(\CC^d)\) can be identified with the nc Hardy space \(\HH^2_d\) via the isometry
    \[
        e_{\s_1} \otimes \cdots \otimes e_{\s_n} \mapsto Z_{\s_1} \cdots Z_{\s_n} = Z^{\s_1 \cdots \s_n}.
    \]
    We consider the setting of \cite{Pop06} of a multi-Toeplitz operator on \(F^2(\CC^d) \otimes \mathcal{E}\) when the auxiliary Hilbert space \(\mathcal{E} = \CC\).
    
    Let \(T\) be a multi-Toeplitz operator on \(\HH_d^2\) and let \(\mu \in \ncm{d}\) be any nc measure with nc Radon--Nikodym derivative \(T\). Popescu \cite{Pop06} defines the prediction entropy of \(T\) to be 
    \[
        e(T) = \log \det \Delta_T
    \]
    where
    \[
        \langle \Delta_T x, x\rangle = \inf\{\langle T(x-p), x-p \rangle : p \in \CC \langle z_1, \ldots, z_d \langle, p(0) = 0\},
    \]
    see Equation (1.1) of \cite{Pop06}.
    It follows that
    \begin{align*}
        \langle \D_T \mathbf{1}, \mathbf{1} \rangle & = \inf\{\langle T(1-p), 1-p \rangle : p \in \CC \langle z_1, \ldots, z_d \langle, p(0) = 0\} \\ 
        & = \inf\{\langle Tq, q \rangle : q \in \CC \langle z_1, \ldots, z_d \langle, q(0) = 1\} \\
        & = \L_{\infty}(0_{1\times1}^{(d)}; \mu).
    \end{align*}

    Finally, we remark that Popescu notes in \cite{Pop06} that this prediction entropy is distinct from the entropy of a multi-Toeplitz kernel considered in the weak Szeg\H{o} limit theorem of \cite{Pop01}.
\end{remark}

\begin{remark}
    Our analogue of Szeg\H{o}'s theorem here comes in two flavours. The first is that the equality (iii) = (iv) is a noncommutative weak Szeg\H{o} limit theorem, in a similar form to that of \cite{Pop01}; the differences between our setting and that one were discussed in \Cref{rem:PopescuWeakSzego}.

    The second is that the equality (v) = (ix) --- or indeed (v) = (vi)' --- shows that \(\prod_{\s \in \FF_d^+} (1 - \lvert \g_\s \rvert^2)\) depends only upon the absolutely continuous part of \(\mu\), precisely as in the case \(d = 1\), and this provides a candidate replacement for the Szeg\H{o} condition \eqref{eqn:SzegoCondition} in \(d\) noncommuting variables: namely,
    \[
        \sum_{\s \in \FF_d^+} \lvert \g_\s \rvert^2 < \infty \quad \text{ if and only if } \quad \log \L_{\infty}(0^{(d)}_{1 \times 1}; \mu) > - \infty.
    \]
\end{remark}

\section{Zeros of Orthonormal Polynomials}

Section 1.7 of \cite{Sim05} studies the location of the zeros of the monic orthogonal polynomials \((\P_n)_{n=0}^{\infty}\) associated to a measure \(\mu\) with an aim towards proving \emph{Verblunsky's theorem}, that if a sequence of complex numbers lies entirely within the open unit disc then there exists a unique measure on the unit circle with that sequence as its Verblunsky coefficients. In particular, Simon provides the following result \cite[Theorem 1.7.1]{Sim05}:
\begin{theorem}[Zeros Theorem]
    Let \(\P_n (z)\) be an OPUC polynomial. Then all the zeros of \(\P_n\) lie in \(\DD\).
\end{theorem}
\begin{remark}
    \label{rem:CommZeros}
    In the proof of this result named \textsc{Third Proof of Theorem 1.7.1}, \cite{Sim05} also proves that the reverse polynomials, there called \(\P_n^*(z)\), are non-vanishing on \(\overline{\DD}\). In other words, Simon shows that if \(z \in \DD\) then \(\P_n^*(z) \neq 0\), and if \(z \in \CC\setminus\DD\), then \(\P_n(z) \neq 0\). It follows that if \(z \in \TT\) then both \(\P_n(z)\) and \(\P_n^*(z)\) are nonzero.
\end{remark}
Indeed, \cite{Sim05} provides six proofs of Theorem 1.7.1, with an historical discussion of the providence of the different ideas involved in each at the end of the section --- we direct the interested reader there for more details.

In the noncommutative setting, the Favard's theorem of \cite{CJ02b} provides the analogous result to Verblunsky's theorem --- see \Cref{thm:NCVerblunsky} for this theorem in terms of nc measures. Nevertheless, the location of the zeros of the monic orthogonal polynomials --- and of the orthonormal polynomials, as \(\p_\s\) and \(\P_\s\) are proportional by definition --- associated to an nc measure remains an interesting question, which we study in this section.

An immediate question presents itself: what is the appropriate notion of ``zero" of a multivariate nc polynomial? The most immediate candidate may be to define a zero of an nc polynomial \(P \in \CC\langle Z \rangle\) as a tuple of matrices \(A\) lying in either the row-ball or the entire nc space \(\CC^d_{\nc} := \dju (\CC^{n\times n})^d\) such that \(P(A) = 0\).

The notion of zero sets of nc polynomials appears already in the literature, in the burgeoning field of noncommutative real algebraic geometry. One such consideration is in \cite{HKV22}, which discusses noncommutative generalisations of Hilbert's Nullstellensatz, and where tuples \(A\) such that \(P(A) = 0\) are called \emph{hard zeros} of \(P\); also in that paper we find the notion of \emph{determinantal zeros}, which shall prove the more useful concept for our study. In \cite{HKV22} we find the following definition.

\begin{definition}
    \label{def:DeterminantalZero}
    Let \(P\in \CC\langle Z \rangle\) be a noncommutative polynomial. A \emph{determinantal zero} of \(P\) is a tuple \(Z \in (\CC^{n\times n})^d\), for any \(n\), such that \(\det P(Z) = 0\). The set of all determinantal zeros of \(P\) is denoted by \(\cZ(P)\) and is given by
    \[
        \cZ(P) = \dju \cZ_n(P) \quad \text{ where } \quad \cZ_n(P) := \left\{Z \in (\CC^{n\times n})^d : \det P(Z) = 0\right\}.
    \]
\end{definition}

\begin{remark}
    As discussed in Section 4, \cite{CJ02b} formulates the same Christoffel-Darboux kernel as we defined above, though we note that (when \(\cH = \CC^n\)) the kernel of \cite{CJ02b} is the restriction of ours to the \(n^{\text{th}}\) level of the row-ball. We recall that \cite[Theorem 4.2]{CJ02b} is an analogue of the Christoffel-Darboux formula for these kernels; though we do not use this formula directly, we shall use the calculations involved in the proof of this result as motivation for our own.
\end{remark}

\begin{remark}
    Though \cite{Sim05} works primarily in terms of the monic orthogonal polynomials, since the orthonormal polynomials are simply rescalings of the monic orthogonal polynomials we have that \(\P_n(z) = 0\) if and only if \(\p_n(z) = 0\) and one can easily rephrase everything in terms of the orthonormal polynomials. On the other hand, the theory we wish to exploit in the noncommutative setting of \cite{CJ02b} is entirely in terms of the orthonormal polynomials \((\p_\s)_{\s \in \FF_d^+}\), so in this section we choose to work with \((\p_\s)_{\s \in \FF_d^+}\) rather than \((\P_\s)_{\s \in \FF_d^+}\) as one might expect from a direct analogy.

    Furthermore, the calculations motivating our starting point involve summing over all words \(\s \in \FF_d^+\) of a fixed length \(n\); as such, the results on \(\p_\s\) and \(\p_\s^{\#}\) we obtain are not for arbitrary \(\s \in \FF_d^+\) but rather for the shortlex-final words of arbitrary length, \((\s(n))_{n=0}^{\infty}\).
\end{remark}

Let \(Z, W \in (\CC^{m\times m})^d\) for some \(m \in \NN\), \(\s \in \FF_d^+\), and \(k \in \{1, \ldots, d\}\). Motivated by a calculation in the proof of \cite[Theorem 4.2]{CJ02b}, the recurrence relations for orthonormal polynomials (\eqref{eqn:recNormal} and \eqref{eqn:recReverseNormal}) provide us with
\begin{align*}
    \p^\#_{k\s}(Z)^*\p^\#_{k\s}(W) - \p_{k\s}(Z)^*\p_{k\s}(W) = \p_{k\s - 1}^\#(Z)^*\p^\#_{k\s - 1}(W) - \p_\s(Z)^* Z_k^* W_k \p_\s(W)
\end{align*}
after some elementary algebra. Following the motivating calculation of \cite{CJ02b}, if we sum these equations (which are distinct for each \(\s \in \FF_d^+\)) over \(k = 1,\ldots, d\) and \(0 \leq \lvert \s \rvert \leq n-1\) we form the equality
\begin{equation} 
    \label{eqn:SummationFormula}
    \p_{\s(n)}^\#(Z)^* \p_{\s(n)}^\#(W) - \sum_{0 \leq \lvert \s \rvert \leq n} \p_\s(Z)^*\p_\s(W) = - \sum_{k=1}^d \sum_{0 \leq \lvert \s \rvert \leq n-1} \p_\s(Z)^* Z_k^* W_k \p_\s(W).
\end{equation}
for each \(n \in \NN\).

We are now in a position to discuss the locations of certain determinantal zeros.

\begin{lemma}
    \label{lem:ZerosOutside}
    Let \(\mu \in \ncm{d}\) be a non-trivial probability nc measure with associated orthonormal polynomials \((\p_\s)_{\s \in \FF_d^+}\). For any \(n \in \NN\), if \(Z \in \CC_\nc^d\) is such that \(\det \p_{\s(n)}^\#(Z^*) = 0\), then \(Z\) lies outside the row-ball.
\end{lemma}
\begin{proof}
    Let \(Z \in \BB_\nc^d\). Setting both \(Z\) and \(W\) in the above to be \(Z^*\), \eqref{eqn:SummationFormula} becomes
    \[
        \p_{\s(n)}^\#(Z^*)^* \p_{\s(n)}^\#(Z^*) - \sum_{0 \leq \lvert \s \rvert \leq n} \p_\s(Z^*)^*\p_\s(Z^*) = - \sum_{k=1}^d \sum_{0 \leq \lvert \s \rvert \leq n-1} \p_\s(Z^*)^* Z_k^* W_k \p_\s(Z^*).
    \]
    We can combine the sums on left and right over \(\lvert \s \rvert = 1, \ldots, n-1\), leaving the terms where \(\s = \emptyset\) and \(\lvert \s \rvert = n\); recalling that \(\p_{\emptyset} = 1\), the result of this manipulation is
    \begin{multline*}
        \p_{\s(n)}^\#(Z^*)^* \p_{\s(n)}^\#(Z^*) = \sum_{\lvert \s \rvert = n} \p_\s(Z^*)^* \p_\s(Z^*) + (I - Z_1Z_1^* - \cdots - Z_dZ_d^*)  \\ + \sum_{1 \leq \lvert \s \rvert \leq n-1} \p_\s(Z^*)^*(I - Z_1Z_2^* - \cdots - Z_dZ_d^*)\p_\s(Z^*).
    \end{multline*}
    Now, \(Z\) is in the row-ball, so \(Z_1Z_1^* + \cdots + Z_dZ_d^* < I\) and hence \(I - Z_1Z_1^* - \cdots - Z_dZ_d^*\) is a positive definite matrix; meanwhile, both
    \[
        \sum_{\lvert \s \rvert = n} \p_\s(Z^*)^* \p_\s(Z^*) \quad  \text{ and } \quad \sum_{1 \leq \lvert \s \rvert \leq n-1} \p_\s(Z^*)^*(I - Z_1Z_2^* - \cdots - Z_dZ_d^*)\p_\s(Z^*)
    \]
    are sums of positive semi-definite matrices, so are themselves positive semi-definite. It follows that \(\p_{\s(n)}^\#(Z^*)^* \p_{\s(n)}^\#(Z^*)\) is positive definite, and therefore has a positive determinant on the row-ball. Finally,
    \[
        \det \p_{\s(n)}^\#(Z^*)^* \p_{\s(n)}^\#(Z^*) = \overline{\det \p_{\s(n)}^\#(Z^*)}\det \p_{\s(n)}^\#(Z^*) = \lvert \det \p_{\s(n)}^\#(Z^*) \rvert^2 > 0
    \]
    so that
    \[
        \det \p_{\s(n)}^\#(Z^*) > 0,
    \]
    and hence any determinantal zero must lie outside the row-ball.
\end{proof}

\begin{remark}
    In one variable, \(\p_{\s(n)}^\#(Z^*)\) is simply \(\p_n^*(\overline{z})\) (where \(\p_n^*\) is the reverse polynomial of the \(n^{\mathrm{th}}\) orthonormal polynomial \(\p_n\)), and in that situation this polynomial has no zeros inside the unit disk. Indeed, this map's zeros are inside the unit disk if and only if those of \(\p_n^*\) are, as \(\lvert \overline{z} \rvert = \lvert z \rvert\); \Cref{lem:ZerosOutside} is then our noncommutative analogue of the statement that reverse polynomials have no zeros inside the disk, with \(\p_{\s(n)}^\#(Z^*)^* \p_{\s(n)}^\#(Z^*)\) taking the role of the modulus-squared of \(\p_n^*\) in our investigation of determinantal zeros.
\end{remark}

Our next consideration is a quantity that reduces to the (modulus-squared of the) \(n^{\mathrm{th}}\) orthonormal polynomial in one variable, and we shall see that it has no zeros \emph{outside} the row-ball; however, to obtain our result we must use a more specialised notion of ``outside the row-ball" than simply \(\CC_\nc^d \setminus \overline{\BB}_\nc^d\).

\begin{lemma}
    \label{lem:ZerosInside}
    Let \(\mu \in \ncm{d}\) be a non-trivial probability nc measure with associated orthonormal polynomials \((\p_\s)_{\s \in \FF_d^+}\). For any \(n \in \NN\), if \(Z \in \CC_\nc^d\) is such that \(\det \sum_{\lvert \s \rvert = n} \p_\s(Z^*)^* \p_\s(Z^*) = 0\), then \(Z\) lies outside the set
    \[
        \mathrm{Ext}(\overline{\BB}^d_{\nc}) := \dju \left\{(Z_1, \ldots, Z_d) \in (\CC^{n \times n})^d : Z_1Z_1^* + \cdots + Z_dZ_d^* > I_n \right\}.
    \]
\end{lemma}
\begin{proof}
    Let \(Z \in \mathrm{Ext}(\overline{\BB}^d_{\nc})\). Again setting \(Z\) and \(W\) in the preamble to be \(Z^*\), when \(Z \in \mathrm{Ext}(\overline{\BB}^d_{\nc})\) we instead have that \(Z_1Z_1^* + \cdots + Z_dZ_d^* - I > 0\), so separating out the \(\lvert \s \rvert = n\) and \(\s = \emptyset\) terms as before we move from \eqref{eqn:SummationFormula} to
    \begin{multline*}
        \sum_{\lvert \s \rvert = n} \p_s(Z^*)^*\p_\s(Z^*) = \p_{\s(n)}^\#(Z^*)^*\p_{\s(n)}^\#(Z^*) + (Z_1Z_1^* + \cdots + Z_dZ_d^* - I) \\ + \sum_{1 \leq \lvert \s \rvert \leq n-1} \p_\s(Z^*)^*(Z_1Z_1^* + \cdots + Z_dZ_d^* - I)\p_\s(Z^*).
    \end{multline*}
    Now, \(Z_1Z_1^* + \cdots + Z_dZ_d^* - I\) is positive definite, and the other two terms on the right hand side here are positive semi-definite, so for \(Z \in \mathrm{Ext}(\overline{\BB}^d_{\nc})\) we have that \(\sum_{\lvert \s \rvert = n} \p_\s(Z^*)^* \p_\s(Z^*)\) is positive definite and therefore has nonzero determinant, that is, its determinantal zeros must lie outside \(\mathrm{Ext}(\overline{\BB}^d_{\nc})\).
\end{proof}

\begin{remark}
    As before, notice that when \(d = 1\), \(\sum_{\lvert \s \rvert = n} \p_\s(Z^*)^* \p_\s(Z^*)\) is simply \(\lvert\p_n(\overline{z})\rvert^2\), and that this (and equivalently \(\p_n\)) has no zeros outside the unit disk, so this quantity is our analogue of (the modulus squared of) \(\p_n\) in the free noncommutative setting.
\end{remark}

Our analogue of the unit circle, as one might imagine after \Cref{lem:ZerosInside}, is the distinguished boundary of the row-ball, the set
\[
    \partial \BB^d_\nc := \dju \left\{(Z_1, \ldots, Z_d) \in (\CC^{n \times n})^d : Z_1Z_1^* + \cdots + Z_dZ_d^* = I_n\right\}
\]
of \emph{row co-isometries}, and this turns out to be a well-behaved set to consider: on \(\partial \BB_\nc^d\), the previous two quantities discussed coincide, so that if one is non-singular on the entirety of this set, we automatically see that the other is too. The argument we provide for non-singularity broadly mimics that of \cite{Sim05}, though some additional care is needed to handle determinantal zeros of a multivariate nc polynomial compared to the zeros of a scalar univariate polynomial.

\begin{lemma}
    \label{lem:ZerosBoundary}
    Let \(\mu \in \ncm{d}\) be a non-trivial probability nc measure with associated orthonormal polynomials \((\p_\s)_{\s \in \FF_d^+}\). For any \(n \in \NN\) and \(Z \in \partial \BB^d_\nc\), we have that 
    \[
        \p_{\s(n)}^\#(Z^*)^* \p_{\s(n)}^\#(Z^*) = \sum_{\lvert \s \rvert = n} \p_\s(Z^*)^*\p_\s(Z^*)
    \]
    and this quantity is non-singular.
\end{lemma}  
\begin{proof}
    Let \(Z \in \partial \BB^d_\nc\); then \(I - Z_1Z_1^* + \cdots + Z_dZ_d^* = 0\), so the same manipulations as in \Cref{lem:ZerosOutside} and \Cref{lem:ZerosInside} this time yield the stated equality immediately. It remains to show that this quantity is non-singular on \(\partial \BB_\nc^d\).
    
    To see this, suppose to the contrary that
    \[
        \det \left(\p_{\s(n)}^\#(W^*)^*\p_{\s(n)}^\#(W^*)\right) = 0
    \]
    for some \(W \in \partial \BB^d_\nc\). Since for any matrix \(A\) we have \(\det AA^* = (\det A) (\overline{\det A}) = \lvert \det A \rvert^2\), this then implies that
    \[
        \det \p_{\s(n)}^\#(W^*) = 0,
    \]
    so in particular if \(W\) is \(m \times m\) then there exists some \(\xi \in \CC^{m}\setminus\{0\}\) such that \(\p_{\s(n)}^\#(W^*)\xi = 0\).

    Recall from the proof of \Cref{lem:ZerosOutside} that if \(Z \in \BB^d_\nc\) then
    \begin{multline*}
        \p_{\s(n)}^\#(Z^*)^* \p_{\s(n)}^\#(Z^*) = \sum_{\lvert \s \rvert = n} \p_\s(Z^*)^* \p_\s(Z^*) + (I - Z_1Z_1^* - \cdots - Z_dZ_d^*)  \\ + \sum_{1 \leq \lvert \s \rvert \leq n-1} \p_\s(Z^*)^*(I - Z_1Z_2^* - \cdots - Z_dZ_d^*)\p_\s(Z^*),
    \end{multline*}
    so that
    \[
        \p_{\s(n)}^\#(Z^*)^* \p_{\s(n)}^\#(Z^*) \geq I - \sum_{k=1}^d Z_kZ_k^*.
    \]
    Notice that if \(r \in (0,1)\) then \(rW \in \BB^d_\nc\), providing us with the matrix inequality
    \begin{align*}
        \p_{\s(n)}^\#(rW^*)^* \p_{\s(n)}^\#(rW^*) & \geq I - \sum_{k=1}^d (rW_k)(rW_k)^* \\
        & = \left(I - \left(\sum_{k=1}^d (rW_k)(rW_k)^*\right)^{1/2}\right)\left(I + \left(\sum_{k=1}^d (rW_k)(rW_k)^*\right)^{1/2}\right) \\
        & = \left(I - \left(r^2\sum_{k=1}^d W_kW_k^*\right)^{1/2}\right)\left(I + \left(r^2\sum_{k=1}^d W_kW_k^*\right)^{1/2}\right);
    \end{align*}
    recalling that \(W \in \partial \BB^d_\nc\), this becomes
    \[
        \p_{\s(n)}^\#(rW^*)^* \p_{\s(n)}^\#(rW^*) \geq (I - rI^{\frac12})(I + rI^{\frac12}) = (1-r)(1+r)I.
    \]
Thus, we have
    \[
        \frac{\xi^* \p_{\s(n)}^\#(rW^*)^* \p_{\s(n)}^\#(rW^*) \xi}{1-r} \geq (1+r)\xi^* \xi > \xi^* \xi > 0.
    \]
    Write the polynomial \(\p_{\s(n)}^\#\) as \(\sum_{\s \in \FF_d^+} \a_\s Z^\s\); with this notation, \(\p_{\s(n)}^\#(W^*)\xi = 0\) implies that 
    \[
        \a_\emptyset \xi = - \sum_{\s \neq \emptyset} \a_\s (W^*)^\s \xi.
    \]
    Substituting this in for the constant coefficient in \(\p_{\s(n)}^\#(rW^*)\xi\) we obtain
    \[
        \p_{\s(n)}^\#(rW^*)\xi = - \sum_{\s \neq \emptyset} \a_\s (W^*)^\s \xi + \sum_{\s \neq \emptyset} \a_\s (rW^*)^\s \xi = \sum_{\s\neq\emptyset} \a_\s (r^{\lvert \s \rvert} - 1) (W^*)^{\s} \xi,
    \]
    which in turns transforms our inequality into
    \begin{align*}
        0 & < \xi^* \xi < \frac{\left(\sum_{\s\neq\emptyset} \a_\s (r^{\lvert \s \rvert} - 1) (W^*)^{\s} \xi\right)^*\left(\sum_{\s\neq\emptyset} \a_\s (r^{\lvert \s \rvert} - 1) (W^*)^{\s} \xi\right)}{1-r} \\
        & = \frac{1}{1-r}\left((1-r)\sum_{\s\neq\emptyset} \a_\s \left(\sum_{j=0}^{\lvert \s \rvert - 1} r^j \right) (W^*)^{\s} \xi\right)^*\left((1-r)\sum_{\s\neq\emptyset} \a_\s \left(\sum_{j=0}^{\lvert \s \rvert - 1} r^j \right) (W^*)^{\s} \xi\right) \\
        & = (1-r)\left(\sum_{\s\neq\emptyset} \a_\s \left(\sum_{j=0}^{\lvert \s \rvert - 1} r^j \right) (W^*)^{\s} \xi\right)^*\left(\sum_{\s\neq\emptyset} \a_\s \left(\sum_{j=0}^{\lvert \s \rvert - 1} r^j \right) (W^*)^{\s} \xi\right) \\
        & \to 0 < \xi^*\xi
    \end{align*}
    as \(r \to 1\), a contradiction. Hence no such \(W\) can exist.
\end{proof}

The following corollary summarises our results on the determinantal zeros of these quantities associated to an nc measure.

\begin{theorem}
    \label{thm:NCZerosThm}
    Let \(\mu \in \ncm{d}\) be a non-trivial probability nc measure with associated orthonormal polynomials \((\p_\s)_{\s \in \FF_d^+}\) and let \(n \in \NN\). Then we have that
    \begin{enumerate}
        \item[\rm{(i)}] if \(Z \in \BB_\nc^d\), then \(\p_{\s(n)}^\#(Z^*)^* \p_{\s(n)}^\#(Z^*)\) is non-singular, i.e.
        \[
            \cZ\left(\p_{\s(n)}^\#(Z^*)^* \p_{\s(n)}^\#(Z^*)\right) \cap \BB_\nc^d = \emptyset;
        \]
        \item[\rm{(ii)}] if \(Z \in \mathrm{Ext}(\overline{\BB}^d_{\nc})\), then \(\sum_{\lvert \s \rvert = n} \p_\s(Z^*)^* \p_\s(Z^*) \) is non-singular, i.e. 
        \[
            \cZ\left(\sum_{\lvert \s \rvert = n} \p_\s(Z^*)^* \p_\s(Z^*)\right) \cap U = \emptyset;
        \]
        \item[\rm{(iii)}] if \(Z \in \partial \BB_{\nc}^d\), then \(\p_{\s(n)}^\#(Z^*)^* \p_{\s(n)}^\#(Z^*)\) and \(\sum_{\lvert \s \rvert = n} \p_\s(Z^*)^* \p_\s(Z^*)\) are non-singular, and furthermore
        \[
            \p_{\s(n)}^\#(Z^*)^* \p_{\s(n)}^\#(Z^*) = \sum_{\lvert \s \rvert = n} \p_\s(Z^*)^* \p_\s(Z^*).
        \]
    \end{enumerate}
\end{theorem}

\begin{proof}
    The first claim is \Cref{lem:ZerosInside}, the second claim is \Cref{lem:ZerosOutside}, and both parts of the final claim are \Cref{lem:ZerosBoundary}.
\end{proof}

\begin{remark}
    We note that items (i) and (ii), as well as the first part of (iii) are highly analogous to the Zeros Theorem \cite[Theorem 1.7.1]{Sim05} and the observations discussed in \Cref{rem:CommZeros}. The analogy to the second part of (iii), that when \(z \in \TT\) we further have that \(\lvert \P_n(z) \rvert = \lvert \P_n^*(z) \rvert\), follows immediately from the formula for the reverse polynomial of a univariate polynomial, though we remark again that this formula is a relationship absent from our setting.
\end{remark}

\bibliographystyle{plain}
\bibliography{bibliography}

\begin{thebibliography}{10}

\bibitem{AMY20}
J.~Agler, J.~E. McCarthy, and N.~J. Young.
\newblock {\em {Operator Analysis: Hilbert Space Methods in Complex Analysis}}.
\newblock Cambridge Tracts in Mathematics. Cambridge University Press, 2020.

\bibitem{ABK02}
D.~Alpay, V.~Bolotnikov, and H.~T. Kaptano\u{g}lu.
\newblock The {S}chur algorithm and reproducing kernel {H}ilbert spaces in the
  ball.
\newblock {\em Linear Algebra Appl.}, 342:163--186, 2002.

\bibitem{BW11}
M.~Bakonyi and H.~J. Woerdeman.
\newblock {\em {Matrix Completions, Moments, and Sums of Hermitian Squares}}.
\newblock Princeton University Press, Princeton, 2011.

\bibitem{BMV16}
J.~A. Ball, G.~Marx, and V.~Vinnikov.
\newblock Noncommutative reproducing kernel {H}ilbert spaces.
\newblock {\em J. Funct. Anal.}, 271(7):1844--1920, 2016.

\bibitem{BMV23}
S.~T. Belinschi, V.~Magron, and V.~Vinnikov.
\newblock Noncommutative {C}hristoffel-{D}arboux kernels.
\newblock {\em Trans. Amer. Math. Soc.}, 376(1):181--230, 2023.

\bibitem{Con96}
T.~Constantinescu.
\newblock {\em {Schur Parameters, Factorization and Dilation Problems }},
  volume~82 of {\em Operator Theory: Advances and Applications}.
\newblock Birkh\"{a}user Verlag, Basel, 1996.

\bibitem{CJ02a}
T.~Constantinescu and J.~L. Johnson.
\newblock {Tensor Algebras and Displacement Structure I. The Schur Algorithm}.
\newblock {\em Zeitschrift für Analysis und ihre Anwendungen}, 21, 03 2002.

\bibitem{CJ02b}
T.~Constantinescu and J.~L. Johnson.
\newblock {Tensor Algebras and Displacement Structure II: Non-Commutative
  Szeg\H{o} Polynomials}.
\newblock {\em Z. Anal. Anwendungen}, 21(3):611--626, 2002.

\bibitem{Con90}
J.~B. Conway.
\newblock {\em {A Course in Functional Analysis}}, volume~96 of {\em Graduate
  Texts in Mathematics}.
\newblock Springer-Verlag, New York, second edition, 1990.

\bibitem{DPS08}
D.~Damanik, A.~Pushnitski, and B.~Simon.
\newblock {The Analytic Theory of Matrix Orthogonal Polynomials}.
\newblock {\em Surv. Approx. Theory}, 4:1--85, 2008.

\bibitem{DK16}
H.~Dym and D.~P. Kimsey.
\newblock C{MV} matrices, a matrix version of {B}axter's theorem, scattering
  and de {B}ranges spaces.
\newblock {\em EMS Surv. Math. Sci.}, 3(1):1--105, 2016.

\bibitem{BH64}
P.R. Halmos and A.~Brown.
\newblock {Algebraic Properties of Toeplitz operators.}
\newblock {\em Journal für die reine und angewandte Mathematik},
  1964(213):89--102, 1964.

\bibitem{Hel02}
J.~W. Helton.
\newblock {``Positive" Noncommutative Polynomials Are Sums of Squares}.
\newblock {\em Annals of Mathematics}, 156(2):675--694, 2002.

\bibitem{HKV22}
J.~W. Helton, I.~Klep, and J.~Vol\v{c}i\v{c}.
\newblock {Factorization of Noncommutative Polynomials and Nullstellensätze
  for the Free Algebra}.
\newblock {\em Int. Math. Res. Not. IMRN}, pages 343--372, 2022.

\bibitem{HMPV09}
J.~W. Helton, S.~McCullough, M.~Putinar, and V.~Vinnikov.
\newblock {Convex Matrix Inequalities Versus Linear Matrix Inequalities}.
\newblock {\em IEEE Trans. Automat. Control}, 54(5):952--964, 2009.

\bibitem{HP57}
E.~Hille and R.~S. Phillips.
\newblock {\em Functional {A}nalysis and {S}emi-{G}roups}, volume Vol. 31 of
  {\em American Mathematical Society Colloquium Publications}.
\newblock American Mathematical Society, Providence, RI, 1957.
\newblock rev. ed.

\bibitem{JM17}
M.~T. Jury and R.~T.~W. Martin.
\newblock {Non-commutative Clark measures for the free and Abelian Toeplitz
  algebras}.
\newblock {\em J. Math. Anal. Appl.}, 456(2):1062--1100, 2017.

\bibitem{JM20}
M.~T. Jury and R.~T.~W. Martin.
\newblock {Column extreme multipliers of the Free Hardy Space}.
\newblock {\em J. Lond. Math. Soc. (2)}, 101(2):457--489, 2020.

\bibitem{JM22a}
M.~T. Jury and R.~T.~W. Martin.
\newblock Fatou's theorem for non-commutative measures.
\newblock {\em Adv. Math.}, 400:Paper No. 108293, 53, 2022.

\bibitem{JM22b}
M.~T. Jury and R.~T.~W. Martin.
\newblock {Lebesgue Decomposition of Non-Commutative Measures}.
\newblock {\em Int. Math. Res. Not. IMRN}, pages 2968--3030, 2022.

\bibitem{JM23}
M.~T. Jury and R.~T.~W. Martin.
\newblock Sub-{H}ardy-{H}ilbert {S}paces in the {N}on-commutative {U}nit {R}ow
  {B}all.
\newblock In {\em Function spaces, theory and applications}, volume~87 of {\em
  Fields Inst. Commun.}, pages 349--398. Springer, Cham, [2023] \copyright
  2023.

\bibitem{JM21}
M.~T. Jury, R.~T.W. Martin, and E.~Shamovich.
\newblock Blaschke–singular–outer factorization of free non-commutative
  functions.
\newblock {\em Advances in Mathematics}, 384:107720, 2021.

\bibitem{KVV14}
D.~S. Kaliuzhnyi-Verbovetskyi and V.~Vinnikov.
\newblock {\em {Foundations of Free Noncommutative Function Theory}}, volume
  199 of {\em Mathematical Surveys and Monographs}.
\newblock American Mathematical Society, Providence, RI, 2014.

\bibitem{Kat04}
Y.~Katznelson.
\newblock {\em {An Introduction to Harmonic Analysis}}.
\newblock Cambridge Mathematical Library. Cambridge University Press, 3
  edition, 2004.

\bibitem{Pau02}
V.~Paulsen.
\newblock {\em {Completely Bounded Maps and Operator Algebras}}, volume~78 of
  {\em Cambridge Studies in Advanced Mathematics}.
\newblock Cambridge University Press, Cambridge, 2002.

\bibitem{Pop01}
G.~Popescu.
\newblock {Structure and Entropy for Positive-Definite Toeplitz Kernels on Free
  Semigroups}.
\newblock {\em J. Math. Anal. Appl.}, 254(1):191--218, 2001.

\bibitem{Pop06}
G.~Popescu.
\newblock {Entropy and Multivariable Interpolation}.
\newblock {\em Mem. Amer. Math. Soc.}, 184(868):vi+83, 2006.

\bibitem{Pop06b}
G.~Popescu.
\newblock Free holomorphic functions on the unit ball of {$B(\mathscr{H})^n$}.
\newblock {\em J. Funct. Anal.}, 241(1):268--333, 2006.

\bibitem{Rud87}
W.~Rudin.
\newblock {\em {Real and Complex Analysis}}.
\newblock McGraw-Hill Book Co., New York, third edition, 1987.

\bibitem{Sch12}
K.~Schm\"{u}dgen.
\newblock {\em {Unbounded Self-adjoint Operators on Hilbert Space}}, volume 265
  of {\em Graduate Texts in Mathematics}.
\newblock Springer, Dordrecht, 2012.

\bibitem{Sch17}
K.~Scmh\"{u}dgen.
\newblock {\em {The Moment Problem}}.
\newblock Springer Cham, 2017.

\bibitem{Sim78}
B.~Simon.
\newblock {A Canonical Decomposition for Quadratic Forms with Applications to
  Monotone Convergence Theorems}.
\newblock {\em J. Functional Analysis}, 28(3):377--385, 1978.

\bibitem{Sim05}
B.~Simon.
\newblock {\em {Orthogonal Polynomials on the Unit Circle: Part 1: Classical
  Theory}}, volume~54 of {\em American Mathematical Society Colloquium
  Publications}.
\newblock American Mathematical Society, Providence, RI, 2005.
\newblock Classical theory.

\bibitem{Sto01}
J.~Stochel.
\newblock Solving the truncated moment problem solves the full moment problem.
\newblock {\em Glasg. Math. J.}, 43(3):335--341, 2001.

\bibitem{Tay37}
A.~E. Taylor.
\newblock {Analytic functions in general analysis}.
\newblock {\em Ann. Scuola Norm. Super. Pisa Cl. Sci. (2)}, 6(3-4):277--292,
  1937.

\bibitem{Tay72}
J.~L. Taylor.
\newblock A general framework for a multi-operator functional calculus.
\newblock {\em Advances in Math.}, 9:183--252, 1972.

\bibitem{Tay73}
J.~L. Taylor.
\newblock Functions of several noncommuting variables.
\newblock {\em Bull. Amer. Math. Soc.}, 79:1--34, 1973.

\bibitem{Voi10}
D.-V. Voiculescu.
\newblock Free analysis questions {II}: the {G}rassmannian completion and the
  series expansions at the origin.
\newblock {\em J. Reine Angew. Math.}, 645:155--236, 2010.

\bibitem{Smu59}
Ju.~L. \v{S}mul\' jan.
\newblock An operator {H}ellinger integral.
\newblock {\em Mat. Sb. (N.S.)}, 49(91):381--430, 1959.

\end{thebibliography}

\end{document}